\documentclass[reqno]{amsart}
\usepackage{dsfont}
\usepackage{enumitem} 
\usepackage{multicol}
\usepackage{amsmath,amssymb,amsthm}
\usepackage[utf8]{inputenc}
\usepackage{aliascnt}
\usepackage{tikz}
\usepackage{mathtools}
\usepackage{multicol}
\usepackage{upgreek}
\usepackage{multirow}

\usepackage{graphicx}
\usepackage{xcolor}
\usepackage{subfigure}

\usepackage[numbers,sort&compress]{natbib}

\usepackage[
bookmarks=true,         
unicode=false,          
pdftoolbar=true,        
pdfmenubar=true,        
pdffitwindow=false,     
pdfstartview={FitH},    
pdftitle={Feedback stabilization for parabolic equations with memory},    
pdfauthor={Author},     
pdfsubject={Subject},   
pdfcreator={Creator},   
pdfproducer={Producer}, 
pdfkeywords={keyword1, key2, key3}, 
pdfnewwindow=true,      
colorlinks=true,       
linkcolor=red,          
citecolor=blue,        
filecolor=magenta,      
urlcolor=cyan,           
hypertexnames=false
]{hyperref}

\usepackage{doi}
\usepackage{verbatim}

\usepackage{setspace}
\theoremstyle{plain}
\newtheorem*{mainresult}{Main Result}

\newtheorem{theorem}{Theorem}[section]

\newtheorem{lemma}[theorem]{Lemma}
\newtheorem{corollary}[theorem]{Corollary}

\newtheorem{assumption}[theorem]{Assumption}

\theoremstyle{definition}

\newtheorem{remark}[theorem]{Remark}

\numberwithin{equation}{section}

\usepackage{geometry}

\newcommand{\linspan}{\mathop{\rm span}\nolimits}

\newcommand{\rest}{\left.\kern-2\nulldelimiterspace\right|_}
\newcommand{\norm}[2]{\left|#1\right|_{#2}}

\newcommand{\Id}{{\mathbf1}}
\newcommand{\indf}{1}

\newcommand{\p}{\partial}

\newcommand{\clA}{{\mathcal A}}

\newcommand{\clC}{{\mathcal C}}

\newcommand{\clF}{{\mathcal F}}

\newcommand{\clI}{{\mathcal I}}

\newcommand{\clK}{{\mathcal K}}
\newcommand{\clL}{{\mathcal L}}

\newcommand{\clU}{{\mathcal U}}
\newcommand{\clV}{{\mathcal V}}
\newcommand{\clW}{{\mathcal W}}

\newcommand{\bbN}{{\mathbb N}}

\newcommand{\bbR}{{\mathbb R}}

\newcommand{\bfA}{{\mathbf A}}

\newcommand{\bfC}{{\mathbf C}}
\newcommand{\bfD}{{\mathbf D}}
\newcommand{\bfE}{{\mathbf E}}
\newcommand{\bfF}{{\mathbf F}}
\newcommand{\bfG}{{\mathbf G}}

\newcommand{\bfK}{{\mathbf K}}
\newcommand{\bfL}{{\mathbf L}}
\newcommand{\bfM}{{\mathbf M}}

\newcommand{\bfR}{{\mathbf R}}
\newcommand{\bfS}{{\mathbf S}}

\newcommand{\bfU}{{\mathbf U}}
\newcommand{\bfV}{{\mathbf V}}
\newcommand{\bfW}{{\mathbf W}}
\newcommand{\bfX}{{\mathbf X}}

\newcommand{\fkA}{{\mathfrak A}}
\newcommand{\fkB}{{\mathfrak B}}

\newcommand{\fkK}{{\mathfrak K}}

\newcommand{\fkM}{{\mathfrak M}}

\newcommand{\fkT}{{\mathfrak T}}

\newcommand{\rmD}{{\mathrm D}}

\newcommand{\rmL}{{\mathrm L}}

\newcommand{\bfn}{{\mathbf n}}

\newcommand{\bfu}{{\mathbf u}}

\newcommand{\rmd}{{\mathrm d}}
\newcommand{\rme}{{\mathrm e}}

\newcommand{\rmm}{{\mathrm m}}

\newcommand{\ovlineC}[1]{\overline C_{\left[#1\right]}}

\definecolor{DarkBlue}{rgb}{0,0.08,0.45}
\definecolor{DarkRed}{rgb}{.65,0,0}
\definecolor{applegreen}{rgb}{0.55, 0.71, 0.0}

\newcounter{mymac@matlab}
\newcommand{\matlab}{MATLAB%
   \ifnum\value{mymac@matlab}<1%
   \textregistered%
   \setcounter{mymac@matlab}{1}%
   \fi%
  }

\newcommand{\Addresses}{{
\footnote{
\hspace{-1.4em}\textsuperscript{1}Department of Mathematics, Indian Institute of Technology Roorkee (IITR),\\
\hspace*{2em}Roorkee,  Uttarakhand 247667, India.\par\nopagebreak
\noindent \textsuperscript{2}Johann Radon Institute for Computational and Applied Mathematics (RICAM), OeAW,\\
\hspace*{2em}Altenbergerstrasse 69, 4040 Linz, Austria.\par\nopagebreak
\noindent  \textit{Emails:} {\small\texttt{arbaz@ma.iitr.ac.in} (A.~Khan); \texttt{sumit{\_}m@ma.iitr.ac.in} (S.~Mahajan);\\
\hspace*{3.7em}\texttt{sergio.rodrigues@ricam.oeaw.ac.at} (S.~Rodrigues)}.\\
\textit{Keywords}: parabolic equations with memory, dynamic output-feedback stabilization, continuous data assimilation, strongly positive kernels,
 explicit feedback-input and output-injection operators.\\
MSC(2020):  93B52, 93C50, 93C05, 93C20,  93D15,
}}}

\begin{document}
\title[Feedback stabilization of parabolic equations with memory]{
Dynamic output-based feedback stabilizability for linear parabolic equations with memory\Addresses}
\author[A. Khan, S. Mahajan, and S.S. Rodrigues]{Arbaz Khan\textsuperscript{1}, Sumit Mahajan\textsuperscript{1} and S\'ergio S. Rodrigues\textsuperscript{2}}

\begin{abstract}
The stabilizability of a general class of linear parabolic equations with a memory term, is achieve by explicit output feedback. The control input is given as a function of a state-estimate provided by an exponential dynamic Luenberger observer based on the output of sensor measurements. The numbers of actuators and sensors are finite. The feedback input and output injection operators are given explicitly involving appropriate orthogonal projections. For exponential kernels, exponential stabilizability can be achieved with the rate of the exponential kernel. The discretization and simulation of the controlled systems are addressed as well and results of simulations are reported showing the performance of the proposed dynamic output-based control feedback input. We include simulations for both exponential and weakly singular Riesz kernels, showing the success of the strategy in obtaining a stabilizing input.
\end{abstract}

\maketitle

\section{Introduction}\label{sec:Intro}
We investigate a class of parabolic equations with memory, defined for time~$t>0$, as
\begin{subequations}\label{sys-y-Intro}
\begin{align}
&\tfrac{\partial}{\partial t}{y}(x,t)= (\nu \Delta -\Id)y(x,t) -  a(x,t) y(x,t) -b(x,t)\cdot\nabla y(x,t)\notag\\
&\hspace{4em} + \eta\int_0^t \fkK(t-s)\Delta y(x,s) \mathrm{d}s +{\textstyle\sum\limits_{j=1}^{M_\sigma}}u_j(t)              \indf_{\omega^j}(x),\\
&\mathfrak{B}y(x,t)|_{\partial\Omega} = 0, \quad \quad y(x,0) = y_0(x),\\
&w(t)=(w_1(t),\dots, w_{S_\varsigma}(t)),\quad\mbox{with}\quad w_i(t)\coloneqq \int_{\omega_{\rm m}^i}y(x,t)\,\rmd\Omega,\quad 1\le i\le S_\varsigma,
\end{align}
\end{subequations}	
where~$\Omega\subset\bbR^d$, with~$d$ a positive integer, is an open convex polygonal/polyhedral subset, the diffusion coefficient is a constant~$\nu>0$, and~$a(x,t)\in\bbR$ and~$b(x,t)\in\bbR^d$ are the reaction and a convection coefficients, repectively. Further, $y_0$ is the initial state, and the linear operator~$\fkB$ defines the conditions on the boundary~$\p\Omega$ of~$\Omega$. Our results are valid for classical boundary conditions of type Dirichlet and Neumann, 
\begin{align}
&\mathfrak{B} = \Id, &&\text{for Dirichlet boundary conditions},\\
&\mathfrak{B} = \bfn\cdot\nabla ,&&\text{for Neumann boundary conditions},
\end{align}
where~$\Id$ denotes the identity operator and~$\bfn=\bfn(\overline x)$ denotes the unit outward normal vector at a point~$\overline x\in\p\Omega$.
In the case the memory term involves the kernel
\begin{align}\label{exk-intro}
\fkK(\tau)\coloneqq\rme^{-\gamma\tau}, \quad 0<\gamma,\tag{\theequation;[\rm sm]}
\end{align}
we are able to show that exponential stabilizability can be achieved with rate~$\gamma$. We shall present the results of simulations confirming this rate. We further present results of simulations for the case of Riesz weakly singular kernels as
\begin{align}\stepcounter{equation}\label{wsk-intro}
\fkK(\tau)\coloneqq\tau^{\gamma-1}, \quad 0<\gamma<1,\tag{\theequation;[\rm ws]}
\end{align}
showing that the strategy is likely successful in this case as well,  in providing an asymptotic decrease of the norm of the state, but where the stability rate is not clear.

The memory term encapsulates information from past history; its coefficient is a nonnegative constant~$\eta>0$.
Finally, $w(t)\in\bbR^{S_\varsigma}$ is the output of measurements of the mean~$w_i(t)$ of the state in given subdomains~$\omega_\rmm^i\subset\Omega$, and $u(t)\in\bbR^{M_\sigma}$ is the control input at time~$t$, which is used to tune a finite number of given actuators represented by indicator functions~$\indf_{\omega^j}$ of given subdomains~$\omega^j\subseteq\Omega$, $1\le j\le M_\sigma$.
	
Note that, we denote the set of actuators and sensors by
\begin{align}\label{UMWS}
U_M\coloneqq\{\indf_{\omega^j}\mid 1\le j\le M_\sigma\},\qquad W_S\coloneqq\{\indf_{\omega_\rmm^j}\mid 1\le j\le S_\varsigma\},
\end{align}
we can write the control and output operators as
\begin{subequations}\label{CtOutOPs}
\begin{align}
&U_M^\diamond\colon\bbR^{M_\sigma}\to \linspan U_M, &&\; U_M^\diamond v\coloneqq{\textstyle\sum\limits_{j=1}^{M_\sigma}}v_j\indf_{\omega^j};\\
&W_S^{\vee}\colon \rmL^2(\Omega)\to \bbR^{S_\varsigma},&&\; W_S^{\vee}h\coloneqq\Bigl( (\indf_{\omega_{\rm m}^1},h)_{\rmL^2(\Omega)},\,\dots,\, (\indf_{\omega_{\rm m}^{S_\varsigma}},h)_{\rmL^2(\Omega)}\Bigr).
\end{align}
\end{subequations}
	
We shall also need the operators~$U_M^\vee$ and~$W_S^{\diamond}$, defined in an analogous way,
\begin{subequations}\label{CtOutOPs-aux}
\begin{align}
&W_S^\diamond\colon\bbR^{S_\varsigma}\to \linspan W_S, &&\; W_S^\diamond v\coloneqq{\textstyle\sum\limits_{j=1}^{S_\varsigma}}v_j\indf_{\omega_{\rmm}^j};\\
&U_M^{\vee}\colon \rmL^2(\Omega)\to \bbR^{M_\sigma},&&\; U_M^{\vee}h\coloneqq\Bigl( (\indf_{\omega^1},h)_{\rmL^2(\Omega)},\,\dots,\, (\indf_{\omega^{M_\sigma}},h)_{\rmL^2(\Omega)}\Bigr).
\end{align}
\end{subequations}

\begin{remark}\label{rem:seqUMWS}
In~\eqref{UMWS}, $M_\sigma$ and~$S_\varsigma$ are positive integers. Our result will hold for large enough numbers~$M_\sigma$ of actuators and~$S_\varsigma$ of sensors. To present/prove the result (cf.~\cite{Rod21-aut}), it will be convenient (cf. Sect.~\ref{ssec:satisf-assKMLS}) to construct the families of actuators and sensors by taking  appropriate subsequences of natural numbers~$M_\sigma\coloneqq\sigma(M)$ and~$S_\varsigma\coloneqq\varsigma(S)$. For the moment, the reader may think of~$M_\sigma=M$ and~$S=S_\sigma$.
\end{remark}
	
\subsection{Dynamic and static output-feedback}\label{ssec:UMdisjWS}
In this manuscript we consider the general case where the measurements are (or may be) taken in subdomains~$\omega_{\rmm}^{j}$ that are different from those~$\omega^{j}$ where the control forcing is acting and construct a {\em dynamic} output-based feedback stabilizing control, depending {\em implicitly} on the output, more precisely, the input depends on a state-estimate provided by a dynamic Luenberger observer, which in turn makes use the available output. The particular case~$U_M=W_S$ shall be shortly addressed in Sect.~\ref{sec:finrem-static-output}, this particular case is interesting because we can construct a {\em static} output-based feedback stabilizing control, depending {\em explicitly} on the output, in other words, in this case the Luenberger observer and corresponding state-estimate is not needed, which will make the computations in real-world applications considerably faster.

\subsection{Dynamic output-based feedback stabilization}\label{ssec:OutFeed-Intro}
We write~\eqref{sys-y-Intro} as an abstract evolutionary equation
\begin{align}\label{sys-y-Intro-A}
\dot{y}&= -A y -  A_{\rm rc} y -\eta\clA_\rmm [y] +U_M^\diamond u,
\qquad y(0) = y_0;\qquad w=W_S^{\vee} y;
\end{align}
with state evolving in the Hilbert pivot space~$H=\rmL^2(\Omega)$. Here~$[y](t)\coloneqq y\rest{(0,t)}=\{y(s)\mid 0< s< t\}$ denotes the history of the state up to time~$t$. The diffusion-like~$A$, reaction--convection-like~$A_{\rm rc}$, memory-like~$\clA_\rmm$,  control~$U_M^\diamond$, and output~$W_S^{\vee}$ operators will be specified later on, together with their required properties.
In particular, this abstract formulation will allow us to simplify the exposition and to derive a result for a class of evolutionary equations, which can potentially be applied to other concrete examples/equations of parabolic type.
Our goal is to find, for a suitable class of so-called positive kernels~$\fkK$, a feedback input operator~$K_M^{[\lambda_1]}\colon y\mapsto u$ (cf.~\cite[Part~I, Sect.~2.5]{Zabczyk08}) so that, system~\eqref{sys-y-Intro-A} is asymptotically stable. For exponential kernels~\eqref{exk-intro}, we will have, for suitable constants~$D_1\ge1$ and~$\mu_1>0$, 
\begin{align}\label{goal-exp-y-Intro}
\norm{y(t)}{H} \le D_1\rme^{-\mu_1t}\norm{y_0}{H}.
\end{align}

Note that we consider~$K_M^{[\lambda_1]}$ depending on the number~$M_\sigma$ of actuators and on a scalar parameter~$\lambda_1\ge0$. The latter will serve as a scaling feedback gain later on.

Let us assume that we have such a stabilizing feedback~$K_M^{[\lambda_1]}$ at our disposal. In real world applications, the state~$y(t)$ may be not available in its entirely, in that case the computation of the input~$u(t)=K_M^{[\lambda_1]}y(t)$ is not possible. However, once we have an estimate~$\widehat y(t)$ for~$y(t)$, we can use the approximation~$\widehat u(t)=K_M^{[\lambda_1]}\widehat y(t)$ for the desired input~$u(t)\in\bbR^M$.
A classical approach is to take the state estimate~$\widehat y(t)$ provided by a Luenberger observer. Here, we  need a finite number of sensors measuring some components of the unknown~$y(t)$ (e.g., the mean/average as in~\eqref{sys-y-Intro}), the output~$z=W_S^{\vee} y(t)\in\bbR^S$ is then injected into the dynamics of the observer so that the provided estimate improves as time increases (cf.~\cite[Part~I, Sect.~2.6]{Zabczyk08}), for example as
\begin{align}\label{goal-exp-esterror-Intro}
\norm{\widehat y(t)-y(t)}{H} \le D_{2} \mathrm{e}^{-\mu_{2} t}\norm{\widehat y_0-y_0}{H}.
\end{align}
for some~$D_2\ge1$ and~$\mu_2>0$. 
Essentially, we arrive at a closed-loop coupled system
\begin{subequations}\label{sys-CL-Intro}
\begin{align}
\dot{y}&= -A y -  A_{\rm rc} y -  \eta\clA_\rmm [y] +U_M^\diamond K_M^{[\lambda_1]} \widehat y,
&&\; y(0) = y_0,\label{sys-CL-Intro-y}\\
\dot{\widehat y}&= -A \widehat y -  A_{\rm rc} \widehat y -   \eta\clA_\rmm [\widehat y] +U_M^\diamond K_M^{[\lambda_1]} \widehat y +L_S^{[\lambda_2]}(W_S^{\vee} \widehat y-W_S^{\vee} y),
&&\; \widehat y(0) = \widehat y_0,\label{sys-CL-Intro-haty}
\end{align}
\end{subequations}	
where~$\widehat y_0$ is an estimate we might have for~$y_0$. Thus, besides the feedback input operator~$K_M^{[\lambda_1]}$, we need to find the output injection operator~$L_S^{[\lambda_2]}$ as well, so that~\eqref{goal-exp-y-Intro} holds true. Again, besides the number~$S$ of sensors, we take the output injection operator depending on a parameter~$\lambda_2\ge0$, which will be used as a scaling factor.
	
By writing~\eqref{sys-CL-Intro} in ``coordinates'' $(z,y)\coloneqq(\widehat y-y,y)$, with~$z_0\coloneqq\widehat y_0-y_0$, we find
\begin{subequations}\label{sys-CL-Intro-error}
\begin{align}
\dot{y}&= -A y -  A_{\rm rc} y -  \eta\clA_\rmm [y] +U_M^\diamond K_M^{[\lambda_1]} y +U_M^\diamond K_M^{[\lambda_1]} z,
&&\quad y(0) = y_0,\label{sys-CL-Intro-error-y}\\
\dot{z}&= -A z -  A_{\rm rc} z -   \eta\clA_\rmm [z]  +L_S^{[\lambda_2]}W_S^{\vee}  z,
&&\quad z(0) = z_0,\label{sys-CL-Intro-error-z}
\end{align}
\end{subequations}	
which we want to make exponentially stable. More precisely, we  look for~$L_S^{[\lambda_2]}$ so that the component~\eqref{sys-CL-Intro-error-z} is   exponentially stable, and for~$K_M^{[\lambda_1]}$ so that the unperturbed version of~\eqref{sys-CL-Intro-y} (i.e., with~$U_M^\diamond K_M^{[\lambda_1]} z=0$) is  asymptotically stable as well.
Such a pair~$(K_M^{[\lambda_1]},L_S^{[\lambda_2]})$ will provide the exponential stability of the closed-loop system as
\begin{align}\label{goal-exp-CLerror-Intro}
&\norm{(z,y)(t)}{H\times H} \le D\rme^{-\mu(t-s)}\norm{(z,y)(s)}{H\times H}.
\end{align}

\begin{mainresult}
For exponential kernels and~$\gamma$ as in~\eqref{exk-intro}, we can find  the feedback-input~$K_M^{[\lambda_1]}$ and output-njection~$L_S^{[\lambda_2]}$ operators, such that the solution of system~\eqref{sys-CL-Intro-error} satisfies~\eqref{goal-exp-CLerror-Intro} with~$\mu=\gamma$, for some~$D\ge1$ and for all~$(z_0,y_0)\in H\times H$.
\end{mainresult}

A more precise statement of this result is given as Theorem~\ref{thm:stab}, which is proven under some assumptions on the families~$\{K_M^{[\lambda]}\mid (M,\lambda)\in\bbN_+\times\overline\bbR_+\}$ and~$\{L_S^{[\lambda]}\mid (S,\lambda)\in\bbN_+\times\overline\bbR_+\}$. The satisfiability of these assumptions is proven in Section~\ref{ssec:satisf-assKMLS}, where these families are given explicitly together with the supports~$\omega^j$ of actuators and~$\omega^j_\rmm$ sensors.
From the Main Result above, the rate~$\mu=\gamma$ given by the exponential kernel can be guaranteed for the exponential stability of the closed-loop system. Simulations will confirm this fact. These simulations also suggest that it may be not possible to improve this rate (e.g., by simply tuning the parameters~$(M,S,\lambda_1,\lambda_2)$ within the used class of feedback-input and output-injection operators as above).
	
\subsection{Literature review}  
This manuscript addresses the stabilization of nonautonomous parabolic equations with memory using  projections-based feedback-input operators. We follow the strategy presented in~\cite{Rod21-aut} (for the memoryless case~$\clA_\rmm=0$), and some developments from~\cite{KunRodWal21}. The feedback stabilization of parabolic equations have been extensively investigated in the literature, but apparently there are no such feedback stabilization results in the literature for dynamics including a memory term.

Projections-based feedbacks, as introduced in~\cite{KunRod19-cocv} in the context of full-state based feedback are given explicitly and can be seen as an alternative to Riccati-based feedbacks.  In~\cite{Rod21-aut} it is shown that a simpler form can be taken which can be used in observer design (or, continuous data assimilation) as well.  We mention also the works~\cite{Barbu13-tac,Barbu12} proposing explicit feedbacks, as alternative to Riccati ones, in the context of autonomous systems and boundary controls. The explicit form of the feedbacks is important for applications where computing the  Riccati based feedbacks is very expensive or unfeasible~\cite[Sect.~7]{Rod23-eect}.
In~\cite{KunRodWal21} some extensions can be found, of the result in~\cite{Rod21-aut}, in the context of full-state stabilization of semilinear systems. For nonlinear dissipative systems we refer to the explicit feedbacks proposed in~\cite{AzouaniTiti14}. 
	
We would like to refer also to~\cite{Munteanu15} where, in the autonomous case, a simple feedback input is designed to stabilize parabolic equations with fading memory by means of controls acting on the boundary of the domain.

The mathematical theory on the stabilization of nonautonomous parabolic-like equations was likely initiated in~\cite{BarbuRodShi11}, in the context of the Navier--Stokes equations. There, the null controllability of the linearized system was explicitly used to derive the Riccati based stabilizability result. Here, instead we use oblique projections based feedbacks, which do not use the exact null controllability of the system. This is important, due to known negative results for the null controllability of the heat equation with memory \cite{GuerreroImanuv13,IvanovPandolfi09}.
We refer the reader to~\cite{KumarPaniJoshi22,LefterLorenzi12} for results on  approximate controllability, and to~\cite{CasasYong23} for results on (finite-horizon) optimal control.

Concerning state estimation, we follow the pioneering studies by Luenberger \cite{Luenberger64,Luenberger66} in the context of finite-dimensional systems. We construct a dynamic Luenberger observer as~\eqref{sys-CL-Intro-haty}; see~\cite{Rod21-aut,AzouaniOlsonTiti14} in the context of parabolic-like memoryless equations. 
	
Concerning numerical simulations, the discretization of parabolic equations with a memory term is an interesting problem, namely, due to the presence of the memory integral term, we refer the reader to the works~\cite{SloanThomee86,LeRouxThomee89,BakaevLarssonThomee98,MahajanKhan24,MahajanKhanMohan24} and references therein. 
	
\subsection{Outline of paper and notation}
In Section~\ref{sec:assum}, general abstract assumptions are made for the operators~$A$, $A_{\rm rc}$, $K$, $\clA_\rmm$, $K_M$, and~$L_S$, defining the dynamics of the closed-loop system.  The proof of the stability of the closed-loop coupled system is given in Section~\ref{sec:stability}. In Section~\ref{sec:satisf-assum}, it is shown that the assumptions made in Section~\ref{sec:assum} are satisfiable, namely, in the context of concrete scalar parabolic equation with memory terms involving positive kernels, and where the actuators and sensors are given by indicator functions of small spatial subdomains.  Sections~\ref{sec:numresults} and~\ref{sec:num-sim-res} are dedicated to the presentation and discussion of several numerical results validating the theoretical findings. In Section~\ref{sec:finrem-static-output}, we address the particular, but important, case where the input depends only (and explicitly) on the output. Finally,  Section~\ref{sec:finalremarks} gathers comments on the theoretical and computational aspects and on potential future works.
	
Throughout this manuscript, we denote the set of real numbers and non-negative integers as $\mathbb{R}$ and $\mathbb{N}$, respectively. Further, we write~$\mathbb{R}_{+}:=(0,+\infty)$ and $\mathbb{N}_{+}:=\mathbb{N} \setminus\{0\}$.

Given Banach spaces $X$ and $Y$, we denote the space of continuous linear mappings from $X$ into $Y$ as $\mathcal{L}(X, Y)$. If $X=Y$, we simplify this notation to $\mathcal{L}(X):=\mathcal{L}(X, X)$. The continuous dual of $X$ is represented by $X^{\prime}:=\mathcal{L}(X, \mathbb{R})$. 
The space of continuous functions from $X$ into $Y$ is denoted by $\mathcal{C}(X, Y)$. The orthogonal complement to a subset $B \subset H$ of a real Hilbert space $H$, with scalar product $(\cdot, \cdot)_{H}$, is denoted $B^{\perp}:=\{h \in H \mid$ $(h, s)_{H}=0$ for all $\left.s \in B\right\}$.

By $\ovlineC{a_{1}, \ldots, a_{n}}$ we denote a non-negative function that increases in each of its non-negative arguments $a_{i}, 1 \leq i \leq n$ and $C, C_{i}, i=0,1, \ldots$ denote generic positive constants that may take different values at different places within the manuscript.

\section{Assumptions}\label{sec:assum}
We present the general assumptions we require for the operators defining the     dynamics of the closed-loop system~\eqref{sys-CL-Intro}.
	
Let~$V$ and~$H=H'$ be two real separable Hilbert spaces.

\begin{assumption}\label{ass:A}		
The inclusion $V\subseteq H$ is dense, continuous, and compact. The operator
$A\in\clL(V,V')$ is symmetric and $(y,z)\mapsto\langle Ay,z\rangle_{V',V}$ is a complete scalar product on~$V.$
\end{assumption}

\begin{assumption}\label{ass:Arc}
For almost every $t>0$, we have~$A_{\mathrm{rc}}(t) \in \mathcal{L}\left(H, V^{\prime}\right)+\mathcal{L}(V, H)$. Further, $\left|A_{\mathrm{rc}}\right|_{\rmL^{\infty}\left(\mathbb{R}_{0}, \mathcal{L}\left(H, V^{\prime}\right)+\mathcal{L}(V, H)\right)}=: C_{\mathrm{rc}}<+\infty$.
\end{assumption}

\begin{assumption}\label{ass:KMLS}
For any given~$\varrho>0$, we have that:
\begin{enumerate}[itemindent=-2em,topsep=5pt,parsep=5pt,partopsep=5pt,leftmargin=2.5em]
\renewcommand{\theenumi}{{\rm(\alph{enumi})}} 
\renewcommand{\labelenumi}{}
\item\theenumi\label{ass:KMLS-a}
There exists~$M_*\in\bbN_+$ such that, for all~$M\ge M_*$ there exists~$\lambda_{1*}=\lambda_{1*}(M)\in\overline\bbR_+$ such that, for all~$\lambda_1\ge \lambda_{1*}$ we have that~$2\langle (A -U_M^\diamond K_M^{[\lambda_1]})w,w\rangle_{V',V}\ge\norm{w}{V}^2+\varrho\norm{w}{H}^2$, for all~$w\in V$;
\item\theenumi\label{ass:KMLS-b} There exists~$S_*\in\bbN_+$ such that, for all~$S\ge S_*$ there exists~$\lambda_{2*}=\lambda_{2*}(S)\in\overline\bbR_+$ such that, for all~$\lambda_{2}\ge \lambda_{2*}$ we have that~$2\langle (A -L_S^{[\lambda_2]}W_S^{\vee})w,w\rangle_{V',V}\ge\norm{w}{V}^2+\varrho\norm{w}{H}^2$, for all~$w\in V$.
\end{enumerate}
\end{assumption}			
		
Finally, we assume that the memory kernel is positive~\cite[Eq.~(2.2)]{MacCamyWong72}\cite[Def.~1]{NohelShea76}.
\begin{assumption}\label{ass:Xi}  
The memory term satisfies, for all~$t>0$ and all~$g\in\rmL^2(0,t;V)$,
\begin{align}\notag
0\le  \int_0^t \langle \clA_\rmm[g](s),g(s)\rangle_{V',V} \mathrm{d}s.
\end{align} 
\end{assumption} 

\begin{remark}
Assumption~\ref{ass:Xi} is enough for showing Lyapunov stabilizability, for kernels as~\eqref{exk-intro} and~\eqref{wsk-intro}. For the particular exponential kernels~\eqref{exk-intro}, we shall show the stronger exponential stabilizability.
\end{remark}

\section{Stability of the closed-loop coupled system}\label{sec:stability}
The following is a more precise statement of the Main Result announced in Section~\ref{ssec:OutFeed-Intro}.
\begin{theorem}\label{thm:stab}
Let Assumptions \ref{ass:A}--\ref{ass:Xi} hold true. Then there exists a pair $(M_*,S_*)\in\bbN_+\times\bbN_+$ such that, for all~$M\ge M_*$ and~$S\ge S_*$ there exist~$(\lambda_{1*},\lambda_{2*})=(\lambda_{1*}(M),\lambda_{2*}(S))\in\overline\bbR_+\times\overline\bbR_+$ such that, for all~$\lambda_1\ge \lambda_{1*}$ and all~$\lambda_2\ge \lambda_{2*}$ we have:
there exists a constant~$D_1\ge0$, such that the weak solution of system~\eqref{sys-CL-Intro-error} satisfies
\begin{subequations}\label{est-fkM-lyap}
\begin{align}
&\left(\norm{(z,y)(t)}{H}^2+\fkM(t)\right)\le (1+D_1)\rme^{-\gamma(t-s)}\left(\norm{(z,y)(s)}{H}^2+\fkM(s)\right).
\end{align}
for all~$(z_0,y_0)\in H\times H$,  with~$\fkM(t)\coloneqq\fkM_z(t)+\fkM_y(t)$ and
\begin{align}
\fkM_g(t)\coloneqq 2\eta\int_0^t\langle \clA_\rmm[g](\tau),g(\tau)\rangle_{V',V}\,\rmd\tau\ge0.
\end{align}
\end{subequations}
In the particular case where the memory term is given by the kernel~\eqref{exk-intro} as
\[
\clA_\rmm[y](t)=\int_0^t \rme^{-\gamma(t-s)}(-\Delta) y(x,s) \mathrm{d}s,
\]
we have that the pair $(M_*,S_*)\in\bbN_+\times\bbN_+$ can be chosen (larger) so that
\begin{subequations}\label{est-fkM-exp}
\begin{align}
&\left(\rme^{2\gamma t}\norm{(z,y)(t)}{H}^2+\overline\fkM(t)\right)\le (1+D_1)\left(\rme^{2\gamma s}\norm{(z,y)(s)}{H}^2+\overline\fkM(s)\right),
\end{align}
where~$\overline\fkM(t)\coloneqq\overline\fkM_z(t)+\overline\fkM_y(t)$ and
\begin{align}
\overline\fkM_g(t)\coloneqq 2\eta\int_0^t\left\langle \int_0^\tau (-\Delta)(\rme^{\frac{\gamma}2 s}g(s))\rmd s,\rme^{\gamma\tau}g(\tau)\right\rangle_{V',V}\rmd\tau\ge0.
\end{align}
\end{subequations}
\end{theorem}
\begin{proof}
Testing the state-estimate error dynamics in~\eqref{sys-CL-Intro-error-z} with~$2 z$, we obtain
\begin{align}
\tfrac{\rmd}{\rmd t}\norm{z}{H}^2&= -2\langle Az -L_S^{[\lambda_2]}W_S^{\vee} z,z\rangle_{V',V}  -  2\langle A_{\rm rc} z +   \eta\clA_\rmm [z], z\rangle_{V',V}\notag\\
&\le-2\langle Az -L_S^{[\lambda_2]}W_S^{\vee} z,z\rangle_{V',V}   +2\ovlineC{C_{\rm rc}}\norm{z}{H}\norm{z}{V}-2\eta\langle  \clA_\rmm [z], z\rangle_{V',V}\label{est-error-1a}
\end{align}
where we used Assumption~\ref{ass:Arc},  which implies that~$2\langle A_{\rm rc} z,  z\rangle_{V',V}\le \ovlineC{C_{\rm rc}}\norm{z}{H}\norm{z}{V}$ (cf.~\cite[Lem.~3.1]{Rod21-sicon}).
By Young inequality, denoting~$\bfA_S^{[\lambda_2]}\coloneqq A -L_S^{[\lambda_2]}W_S^{\vee}$, we find
\begin{align}
\tfrac{\rmd}{\rmd t}\norm{z}{H}^2&\le -2\langle \bfA_S^{[\lambda_2]} z,z\rangle_{V',V} +\ovlineC{C_{\rm rc}}^2\norm{z}{H}^2 +\norm{z}{V}^2 -2\eta\langle   \clA_\rmm [z],  z\rangle_{V',V}.\label{est-error-1b}
\end{align}

Now, we fix an arbitrary constant~$\overline\mu>0$. By using Assumption~\ref{ass:KMLS}\ref{ass:KMLS-b}, with~$\varrho=2\overline\mu+\ovlineC{C_{\rm rc}}^2$, we can find~$S_*\in\bbN_+$, such that for any given~$S\ge S_*$, we can find~$\lambda_{2*}=\lambda_{2*}(S)\in\overline\bbR_+$ such that, for all~$\lambda_{2}\ge \lambda_{2*}$ we have that
\begin{align}
\tfrac{\rmd}{\rmd t}\norm{z}{H}^2&\le  -2\overline\mu\norm{z}{H}^2-2\eta\langle   \clA_\rmm [z],  z\rangle_{V',V}.\label{est-error-1c}
\end{align}

Now, time integration over~$(s,t)$ gives us
\[
\norm{z(t)}{H}^2-\norm{z(s)}{H}^2\le  -2\overline\mu\int_s^t\norm{z(\tau)}{H}^2\,\rmd\tau-2\eta\int_s^t\langle   \clA_\rmm [z],  z\rangle_{V',V}\,\rmd\tau\le-\fkM_z(t)+\fkM_z(s),
\]
with~$\fkM_z(t)$ as in~\eqref{est-fkM-lyap}, which we write as
\begin{equation}
\norm{z(t)}{H}^2+\fkM_z(t)+2\overline\mu\int_s^t\norm{z(\tau)}{H}^2\,\rmd\tau\le\norm{z(s)}{H}^2+\fkM_z(s).\label{est-z-lyap}
\end{equation}

By testing~\eqref{sys-CL-Intro-error-y} with~$2 y$ we obtain the analogue of~\eqref{est-error-1b},
\begin{align}
\tfrac{\rmd}{\rmd t}\norm{y}{H}^2&\le -2\langle \bfA_M^{[\lambda_1]}y,y\rangle_{V',V} +\ovlineC{C_{\rm rc}}^2\norm{y}{H}^2 -2{\eta}\langle   \clA_\rmm [y],  y\rangle_{V',V}+2(U_M^\diamond K_M^{[\lambda_1]} z, y)_{H}.\hspace{-2em}
\end{align}
Hence, with~$\bfA_M^{[\lambda_1]}\coloneqq A -U_M^\diamond K_M^{[\lambda_1]} $ instead of~$A-L_S^{[\lambda_2]}W_S^{\vee}$,  and with the extra term~$2(U_M^\diamond K_M^{[\lambda_1]} z, y)_{H}$, which we estimate as follows,
\begin{align}
2(U_M^\diamond K_M^{[\lambda_1]} z,y)_{H}
&\le2\mu\norm{y}{H}^2 +2^{-1}\mu^{-1}\norm{U_M^\diamond K_M^{[\lambda_1]} z}{H}^2\le2\epsilon\norm{y}{H}^2 +\epsilon^{-1}C_1\norm{z}{H}^2\label{est-y-3}
\end{align}
for all~$\epsilon>0$, with $C_1\coloneqq2^{-1}\norm{U_M^\diamond K_M^{[\lambda_1]}}{\clL(H)}^2$. Note that, we can repeat the arguments above, simply by using Assumption~\ref{ass:KMLS}\ref{ass:KMLS-a} instead of~\ref{ass:KMLS}\ref{ass:KMLS-b}, to conclude that we can find~$M_*\in\bbN_+$, such that for any~$M\ge M_*$, we can find~$\lambda_{1*}=\lambda_{1*}(M)\in\overline\bbR_+$ such that, for all~$\lambda_{1}\ge \lambda_{1*}$ we have
the analogue of~\eqref{est-error-1c} as
\begin{align}
\tfrac{\rmd}{\rmd t}\norm{y}{H}^2&\le  -2\overline\mu\norm{y}{H}^2-2{\eta}\langle   \clA_\rmm [y],  y\rangle_{V',V}+2\epsilon\norm{y}{H}^2 +\epsilon^{-1}C_1\norm{z}{H}^2\notag\\
&\le  -2(\overline\mu-\epsilon)\norm{y}{H}^2-2{\eta}\langle   \clA_\rmm [y],  y\rangle_{V',V} +\epsilon^{-1}C_1\norm{z}{H}^2,\label{est-y-1c}
\end{align}
In particular, choosing~$\epsilon< \overline\mu$ and denoting~$\mu\coloneqq\overline\mu-\epsilon$, by time integration we find
\[
\norm{y(t)}{H}^2-\norm{y(s)}{H}^2\le  -2\mu\int_s^t\norm{y(\tau)}{H}^2\,\rmd\tau+(\overline\mu-\mu)^{-1}C_1\int_s^t\norm{z(\tau)}{H}^2\,\rmd\tau-\fkM_y(t)+\fkM_y(s),
\]
with~$\fkM_y(t)$ as in~\eqref{est-fkM-lyap}. After reordering and using~\eqref{est-z-lyap}, we obtain
\begin{align}
&\norm{y(t)}{H}^2+\fkM_y(t)+2\mu\int_s^t\norm{y(\tau)}{H}^2\,\rmd\tau\notag\\
&\hspace{3em}\le  \norm{y(s)}{H}^2+\fkM_y(s)+C_1(\overline\mu-\mu)^{-1}(2\overline\mu)^{-1}\left(\norm{z(s)}{H}^2+\fkM_z(s)\right).\label{est-y-lyap}
\end{align}
By combining~\eqref{est-y-lyap} with~\eqref{est-z-lyap} we arrive at~\eqref{est-fkM-lyap} with
$D_1\coloneqq 1+C_1(\overline\mu-\mu)^{-1}(2\overline\mu)^{-1}$.

To show~\eqref{est-fkM-exp}, we restart from~\eqref{est-error-1b}, and using Assumption~\ref{ass:KMLS}\ref{ass:KMLS-b}, again with the larger~$\varrho=\gamma+2\overline\mu+\ovlineC{C_{\rm rc}}^2$. In this way we arrive at the analogue of~\eqref{est-z-lyap} as follows, for~$\varphi_z(t)\coloneqq\rme^{2\gamma t}\norm{z(t)}{H}^2$,
\begin{align}
\tfrac{\rmd}{\rmd t}\varphi_z&=\gamma\varphi_z+\rme^{2\gamma t}\tfrac{\rmd}{\rmd t}\norm{z(t)}{H}^2\le  -2\overline\mu\varphi_z(t)-2\eta\langle   \clA_\rmm [z](t),  \rme^{2\gamma t}z(t)\rangle_{V',V}.\label{est-error-exp}
\end{align}
Now, time integration gives
\[
\varphi_z(t)-\varphi_z(s)+2\mu\int_s^t\norm{\varphi_z(\tau)}{H}^2\,\rmd\tau\le -2\eta\int_s^t\langle\rme^{2\gamma \tau}z(\tau),   \clA_\rmm [z](\tau)\rangle_{V',V}\,\rmd \tau.
\]
and, in the particular case of the exponential kernel we find
\begin{align*}
&2\eta\int_s^t\langle\rme^{2\gamma \tau}z(\tau),   \clA_\rmm [z](\tau)\rangle_{V',V}\,\rmd \tau=2\eta\int_s^t\langle\rme^{2\gamma \tau}z(\tau), \int_0^\tau  \rme^{-\gamma(\tau-r)}Az(r)\rangle_{V',V}\,\rmd r\,\rmd \tau\\
&\hspace{3em}=2\eta\int_s^t\langle\rme^{\gamma \tau}z(\tau), \int_0^\tau  A(\rme^{\gamma r}z(r))\rangle_{V',V}\,\rmd r\,\rmd \tau
=\overline\fkM_z(t)-\overline\fkM_z(s),
\end{align*}
with~$\overline\fkM_z(t)$ as in~\eqref{est-fkM-exp}, which leads to the analogue of~\eqref{est-z-exp},
\begin{align}
\varphi_z(t)+\overline\fkM_z(t)+2\mu\int_s^t\norm{\varphi_z(\tau)}{H}^2\,\rmd\tau\le \varphi_z(s)+\overline\fkM_z(s).\label{est-z-exp}
\end{align}
Note that~$\overline\fkM_z(t)\ge0$ (because the constant kernel~$\fkK(t)=1$ is positive; cf.~Rem.~\ref{rem:constker}).

Arguing as above, with~$\varphi_y(t)\coloneqq\rme^{2\gamma t}\norm{y(t)}{H}^2$, will give us the analogue of~\eqref{est-y-lyap},
\begin{align}
&\varphi_y(t)+\overline\fkM_y(t)+2\mu\int_s^t\norm{\varphi_y(\tau)}{H}^2\,\rmd\tau\notag\\
&\hspace{3em}\le \varphi_y(s)+\overline\fkM_y(s)+C_1(\overline\mu-\mu)^{-1}(2\overline\mu)^{-1}\left(\norm{\varphi_z(s)}{H}^2+\overline\fkM_z(s)\right),\label{est-y-exp}
\end{align}
which together with~\eqref{est-z-exp}, allow us to obtain~\eqref{est-fkM-exp}.
\end{proof}

\section{Satisfiability of the assumptions}\label{sec:satisf-assum}
We apply the abstract result in Section~\ref{sec:stability} to the concrete system~\eqref{sys-y-Intro}. For that we write~\eqref{sys-y-Intro} in abstract form and show that the assumptions required in Section~\ref{sec:assum} are satisfied.

\subsection{Satisfiability of Assumptions~\ref{ass:A} and~\ref{ass:Arc}}\label{ssec:satisf-assAArc}	
We take the pivot space~$H=\rmL^2(\Omega)$ and the operator~$A\colon V\to V'$ as the scaled and shifted Laplacian in~$H$, under the given boundary conditions, as follows
\begin{align}\label{def.A}
\langle A y,z\rangle_{V',V}:=\nu(\nabla y,\nabla z)_{H^d}+(y,z)_{H}, \qquad\mbox{for}\quad (y,z)\in V\times V,
\end{align}
where $V=V_\fkB$ (depending on boundary conditions) is given by the Sobolev subspace
\begin{align}
&V = \{h\in W^{1,2}(\Omega)\mid h\rest{\p\Omega}=0\}, &&\text{for Dirichlet boundary conditions};\notag\\
&V = W^{1,2}(\Omega) ,&&\text{for Neumann boundary conditions}.\notag
\end{align}
We recall that the domain of~$A$ is given by
\begin{align}
&\rmD(A) = \{h\in W^{2,2}(\Omega)\mid h\rest{\p\Omega}=0\}, &&\text{for Dirichlet boundary conditions};\notag\\
&\rmD(A) = \{h\in W^{2,2}(\Omega)\mid \bfn\cdot\nabla h\rest{\p\Omega}=0\},&&\text{for Neumann boundary conditions}.\notag
\end{align}

We take the reaction and convection terms essentially bounded, 
\begin{align}\notag
a\in \rmL^\infty(\bbR_+\times\Omega)\quad\mbox{and}\quad b\in \rmL^\infty(\bbR_+\times\Omega)^d.
\end{align}

It is straightforward to check that Assumptions~\ref{ass:A} and~\ref{ass:Arc} are satisfied. In particular, with~$C_{\rm rc}=\ovlineC{\norm{a}{\rmL^\infty(\bbR_+\times\Omega)},\norm{b}{\rmL^\infty(\bbR_+\times\Omega)^d}}$.

\subsection{Satisfiability of Assumption~\ref{ass:KMLS}}\label{ssec:satisf-assKMLS}
We take  actuators and sensors as indicator functions of small subdomains as illustrated in Fig.~\ref{fig:suppActSensGrid}, combined in a chessboard pattern; see~\cite[Fig.~1]{Rod21-aut}. More precisely,
\begin{figure}[htbp]%
\centering%
\includegraphics[width=1\textwidth]{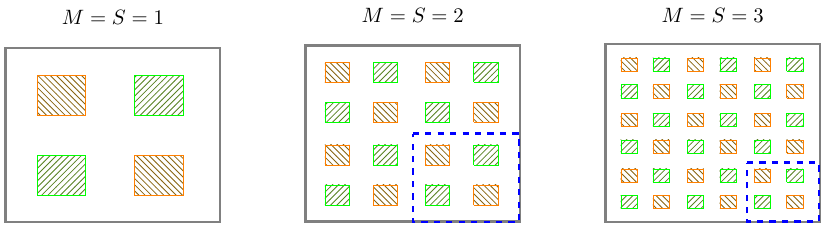}
\caption{Supports of actuators (slash-$\slash$-lines) and sensors (backslash-$\backslash$-lines).} \label{fig:suppActSensGrid}
\end{figure}
we consider  sequences of families~$(U_M)_{M\in\bbN_+}$  of actuators and~$(W_S)_{S\in\bbN_+}$ of sensors (cf.~\eqref{UMWS} and Rem.~\ref{rem:seqUMWS}) as
\begin{align}\label{ActSens}
U_M\coloneqq \{\indf_{\omega^j}\mid 1\le j\le M_\sigma\},\qquad W_S\coloneqq \{\indf_{\omega_\rmm^j}\mid 1\le j\le S_\varsigma\},
\end{align}
with supports as in Fig.~\ref{fig:suppActSensGrid} (in the case~$d=2$), thus with~$M_\sigma=4 M^d$ and~$M_\sigma=4S^d$. Essentially, the case~$M=1$ includes two actuators and  two sensors and the case~$M>1$ is the union of~$M^d$ rescaled copies of the case~$M=1$; we can see 
one of such copies underlined with a dashed line at the right-bottom corner. The same rescaling-and-copy strategy can be applied to general polygonal domains~$\Omega$ (cf.~\cite[Rem.~2.8]{AzmiKunRod23}, \cite[Sect.~3]{EdelsbrunnerGrayson00},  \cite[Thm.~4.1]{Bey00}).  Thus the arguments in~\cite{KunRodWal21} can be followed to derive the analogue of~\cite[Lem.~3.5]{KunRodWal21}
for both sensors and actuators, as follows. Let~$(\widetilde U_M)_{M\in\bbN_+}$ and~$(\widetilde W_S)_{S\in\bbN_+}$ be sequences of families of auxiliary functions~$\widetilde U_M=\{\Phi_j\mid 1\le j\le M_\sigma\}$ and~$\widetilde W_S=\{\Psi_j\mid 1\le j\le S_\varsigma\}$ so that the direct sums
\begin{equation}\notag
\widetilde \clU_M\oplus  \clU_M^\perp=H=\widetilde \clW_S\oplus  \clW_S^\perp
\end{equation}
hold true, 
where $\clU_M\coloneqq\linspan U_M$; $\widetilde \clU_M\coloneqq\linspan \widetilde U_M$; $\clW_S\coloneqq\linspan W_S$; $\widetilde \clW_S\coloneqq\linspan \widetilde W_S$.
Then, taking auxiliary functions as ``regularized indicator functions'' as in~\cite[Sect.~6]{Rod21-aut} we can see that the assumptions in~\cite[Assum.~2.4 and~2.5]{Rod21-aut} hold true. Then we can follow the arguments in~\cite{KunRodWal21} and arrive at the analogue of~\cite[Lem.~3.5]{KunRodWal21} as follows. 
\begin{lemma}\label{lem:actsens}
For every constant $\zeta> 0$, we can find~$M\in\bbN_+$, $S\in\bbN_+$, $\lambda_1=\lambda_1(M)\in\overline\bbR_+$ and~$\lambda_2=\lambda_2(S)\in\overline\bbR_+$ large enough such that
\begin{align}
\norm{y}{V}^2+2\lambda_1\norm{P_{\widetilde \clU_M}^{\clU_M^\perp}y}{H}^2\ge\zeta\norm{y}{H}^2\quad\mbox{and}\quad
\norm{y}{V}^2+2\lambda_2\norm{P_{\widetilde \clW_S}^{\clW_S^\perp}y}{H}^2\ge\zeta\norm{y}{H}^2,\notag
\end{align}
for all~$y\in V$. Further, $M=\ovlineC{\zeta}$, $S=\ovlineC{\zeta}$, $\lambda_1=\ovlineC{\zeta,\beta_{1,M}}$, and~$\lambda_2=\ovlineC{\zeta,\beta_{2,S}}$, with
\begin{align}
\beta_{1,M}\coloneqq\max_{\phi\in\widetilde\clU_M}\tfrac{\norm{\phi}{V}}{\norm{\phi}{H}}\quad\mbox{and}\quad
\beta_{2,S}\coloneqq\max_{\phi\in\widetilde\clW_S}\tfrac{\norm{\phi}{V}}{\norm{\phi}{H}}.\notag
\end{align}
\end{lemma}

Furthermore, we have the analogue of~\cite[Cor.~3.1]{KunRodWal21} as well, as follows. 
\begin{lemma}\label{lem:actsens-orth}
For every constant $\zeta> 0$, we can find~$M\in\bbN_+$, $S\in\bbN_+$, $\lambda_1=\lambda_1(M)\in\overline\bbR_+$ and~$\lambda_2=\lambda_2(S)\in\overline\bbR_+$ large enough such that
\begin{align}
\norm{y}{V}^2+2\lambda_1\norm{P_{\clU_M}^{\clU_M^\perp}y}{H}^2\ge\zeta\norm{y}{H}^2\quad\mbox{and}\quad
\norm{y}{V}^2+2\lambda_2\norm{P_{\clW_S}^{\clW_S^\perp}y}{H}^2\ge\zeta\norm{y}{H}^2,\notag
\end{align}
for all~$y\in V$. Furthermore, $M=\ovlineC{\zeta}$, $S=\ovlineC{\zeta}$, $\lambda_1=\ovlineC{\zeta,p_M\beta_{1,M}}$, and~$\lambda_2=\ovlineC{\zeta,p_S\beta_{2,S}}$, with~$\beta_{1,M}$ and~$\beta_{2,S}$ as in Lemma~\ref{lem:actsens} and with~$p_M\coloneqq\norm{P_{\widetilde\clU_M}^{\clU_M^\perp}}{\clL(H,H)}$.
\end{lemma}

The next result tells us that Assumption~\ref{ass:KMLS} is satisfied for actuators and sensors~\eqref{ActSens} with supports as in Fig.~\ref{fig:suppActSensGrid} and with the feedback input operators:
\begin{subequations}\label{KL-Oper}
\begin{align}
\mbox{either}\quad&K_M^{[\lambda_1]}\coloneqq -\lambda_1[\clV_U]^{-1}U_M^\vee P_{ \clU_M}^{\widetilde\clU_M^\perp}P_{\widetilde \clU_M}^{\clU_M^\perp}
\quad\mbox{or}\quad K_M^{[\lambda_1]}\coloneqq -\lambda_1[\clV_U]^{-1}U_M^\vee P_{ \clU_M}^{\clU_M^\perp};
\intertext{and with the output injection operators:}
\mbox{either}\quad&L_S^{[\lambda_2]}\coloneqq -\lambda_2P_{ \clW_S}^{\widetilde\clW_S^\perp} P_{\widetilde \clW_S}^{\clW_S^\perp}W_S^\diamond[\clV_W]^{-1}
\quad\mbox{or}\quad L_S^{[\lambda_2]}\coloneqq -\lambda_2W_S^\diamond[\clV_W]^{-1};
\end{align}
\end{subequations}
where~$[\clV_W] \in\bbR^{S_\varsigma\times S_\varsigma}$ is the matrix with entries~$[\clV_W]_{(i,j)}\coloneqq (\indf_{\omega_\rmm^j},\indf_{\omega_\rmm^i})_{H}$ in the $i$th row and~$j$th column; analogously, $[\clV_U] \in\bbR^{M_\sigma\times M_\sigma}$ is the matrix with entries~$[\clV_U]_{(i,j)}\coloneqq (\indf_{\omega^j},\indf_{\omega^i})_{H}$ in the $i$th row and~$j$th column. Observe also that (cf.~\cite[Sect.~1.2]{Rod21-aut})
\begin{align}\label{output-Proj}
U_M^\diamond [\clV_U]^{-1}U_M^\vee=P_{\clU_M}^{\clU_M^\perp}\quad\mbox{and}\quad W_S^\diamond [\clV_W]^{-1}W_S^\vee=P_{\clW_S}^{\clW_S^\perp}.
\end{align}
\begin{corollary}
Assumption~\ref{ass:KMLS} holds true for families of actuators and sensors as in Lemma~\ref{lem:actsens}, and the operators as in~\eqref{KL-Oper}.
\end{corollary}
\begin{proof}
Let~$\varrho>0$. By~\eqref{output-Proj}, for the feedback inputs we find, with~$\clF_M\in\{\clU_M,\widetilde \clU_M\}$,
\begin{align}\notag
2\langle (A -U_M^\diamond K_M^{[\lambda_1]})w,w\rangle_{V',V}=2\norm{y}{V}^2+2\lambda_1\norm{P_{\clF_M}^{\clU_M^\perp}y}{H}^2= \norm{y}{V}^2+\norm{y}{V}^2+2\lambda_1\norm{P_{\clF_M}^{\clU_M^\perp}y}{H}^2
\end{align}
and by Lemmas~\ref{lem:actsens} and~\ref{lem:actsens-orth} we find that Assumption~\ref{ass:KMLS}\ref{ass:KMLS-a} follows.
For the output injection, we can use the same arguments, simply recalling that (cf.~\cite[Eq.~(1.19a)]{Rod21-aut})
\begin{align}\label{ProjOb=ObOr}
 P_{\widetilde\clW_S}^{\clW_S^\perp}=P_{\widetilde\clW_S}^{\clW_S^\perp}P_{\clW_S}^{\clW_S^\perp}.
\end{align}
Therefore, we conclude that Assumption~\ref{ass:KMLS}\ref{ass:KMLS-b} holds true as well.
\end{proof}

\begin{remark}
We can take variations of the feedback-input and output-injection operators as~$K_M^{[\lambda_1,1]}= -\lambda_1P_{\clU_M}^{\widetilde\clU_M^\perp}A P_{\widetilde \clU_M}^{\clU_M^\perp}$ and~$L_S^{[\lambda_2,1]}\coloneqq -\lambda_2P_{ \clW_S}^{\widetilde\clW_S^\perp} AP_{\widetilde \clW_S}^{\clW_S^\perp}[\clV_W]^{-1}W_S^\vee$,  using the diffusion operator~$A$, which allow for a choice of the~$(\lambda_1,\lambda_2)$ independently of~$(M,S)$. This can be seen by following arguments as in~\cite[Proof of Thm.~4.2]{Rod21-aut}.  
\end{remark}

\subsection{Satisfiability of Assumption~\ref{ass:Xi}}\label{ssec:satisf-assXi}
We consider scalar kernels satisfying
\begin{subequations}\label{parab-reg-kernels}
\begin{align}
&\fkK\in\clC^2(0,\infty){\,\textstyle\bigcap\,}\rmL^1(0,1),\\\
&\fkK(t)\ge0\ge\dot\fkK(t),\qquad\ddot\fkK(t)\ge0,\quad\mbox{for all}\quad t>0.
\end{align}
\end{subequations}
Here,~$\fkK=0$ corresponds to the memoryless case. 
The memory operator is taken as
\begin{equation}\label{mem-gamma-Lap}
\clA_{\rmm}[y](\tau)\coloneqq\int_0^\tau\fkK(\tau-r)(-\Delta)  y(r)\,\rmd r,
\end{equation}
as in~\eqref{sys-y-Intro} with~$\gamma\in(0,1)$.  By direct computations,
\begin{align}\notag
&\int_{0}^t \langle \clA_\rmm[y](\tau),y(\tau)\rangle_{V',V}\,\rmd \tau=\int_0^t\left\langle\int_0^\tau\fkK(\tau-r) (-\Delta) y(r)\,\rmd r,y(\tau)\right\rangle_{V',V}\rmd\tau\notag\\
&\hspace{0em}=\int_0^t\left(\int_0^\tau\fkK(\tau-r) \left\langle-\Delta y(r),y(\tau)\right\rangle_{V',V}\,\rmd r\right)\rmd\tau\notag\\
&\hspace{0em}=\int_0^t\left(\int_0^\tau\fkK(\tau-r) \left(\nabla y(r),\nabla y(\tau)\right)_{H^d}\,\rmd r\right)\rmd\tau
\notag\\ &\hspace{0em}=\sum_{i=1}^d\int_\Omega\int_0^t\int_0^\tau\fkK(\tau-r) \tfrac{\p}{\p x_i} y(x,r)\tfrac{\p}{\p x_i} y(x,\tau)\,\rmd r~\rmd\tau~\rmd x.\label{monot-xiLap}
\end{align}
Recall that~$\fkK$ as in~\eqref{mem-gamma-Lap}, defines a  positive kernel~\cite[Cor.~2.2 and Eq.~(1.1]{NohelShea76}, that is,
\begin{align}\label{posKern_powgamma-1}
&\int_0^tf(\tau)\int_0^\tau\fkK(\tau-r) f(r) \,\rmd r\,\rmd\tau
\ge  0,\quad\mbox{for all } t>0 \mbox{ and all } f\in \rmL^2(0,t),
\end{align}
which combined with~\eqref{monot-xiLap}, shows that Assumption~\ref{ass:Xi} is satisfied.

\begin{remark}\label{rem:constker}
The result in~\cite[Cor.~2.2]{NohelShea76} is stated for nonconstant~$\fkK$, and is intended to the more restrictive class of ``strongly'' positive kernels. By direct calculations, we find that a  constant kernel~$\fkK(t)=\kappa\ge0$ is a positive kernel as well:
\[
2\int_0^tf(\tau)\int_0^\tau \fkK(\tau-r) \,\rmd r\,\rmd\tau=\kappa\int_0^t\tfrac{\rmd}{\rmd \tau}\left(\int_0^\tau f(r) \,\rmd r\right)^2\rmd\tau=\kappa\left(\int_0^t f(r) \,\rmd r\right)^2\rmd\tau
\ge  0.
\]
\end{remark}

\section{Numerical Results}\label{sec:numresults}
In Section~\ref{sec:satisf-assum} we have seen that the abstract assumptions are satisfied for scalar parabolic systems as~\eqref{sys-y-Intro}, with scalar kernels~$\fkK$ as in~\eqref{parab-reg-kernels}, and with feedback input and output injection operators as in~\eqref{KL-Oper} with indicator functions actuators and sensors as in Fig.~\ref{fig:suppActSensGrid}.
Here, we present the results of simulations concerning the output-based feedback stabilization of a system as~\eqref{sys-y-Intro} for kernels as~\eqref{wsk-intro} and~\eqref{exk-intro};  these kernels satisfy~\eqref{parab-reg-kernels}.
We consider the system
\begin{subequations}\label{sys-num}
\begin{align}
\tfrac{\partial}{\partial t} {y}&=-(-\nu\Delta+\Id)y  -  a y - b \cdot\nabla y  -\eta \int_0^t\fkK(t-s)(-\Delta)y(s)ds\notag\\
&\quad +U_M^\diamond \Bigl(- \lambda_1[\clV_U]^{-1}U_M^\vee P_{\clU_M}^{\widetilde\clU_M^\perp}P_{\widetilde\clU_M}^{\clU_M^\perp}\widehat y\Bigr),\label{sys-num-y}\\
\tfrac{\partial}{\partial t} {\widehat y}&=-(-\nu\Delta+\Id)\widehat y  -  a  \widehat y - b \cdot\nabla \widehat y  -\eta \int_0^t\fkK(t-s)(-\Delta)\widehat y(s)ds \notag\\
&\quad+U_M^\diamond \Bigl(- \lambda_1[\clV_U]^{-1}U_M^\vee P_{\clU_M}^{\widetilde\clU_M^\perp}P_{\widetilde\clU_M}^{\clU_M^\perp}\widehat y\Bigr) - \lambda_2 P_{\clW_S}^{\widetilde\clW_S^\perp}P_{\widetilde \clW_S}^{\clW_S^\perp}W_S^\diamond [\clV_W]^{-1} W_S^\vee(\widehat y-y).\label{sys-num-haty}
\end{align}
\end{subequations}	
The dynamics~\eqref{sys-num-y} of the controlled state~$y$ is subject to the feedback input control
\begin{align}
u(t)\coloneqq- \lambda_1[\clV_U]^{-1}U_M^\vee P_{\clU_M}^{\widetilde\clU_M^\perp}P_{\widetilde\clU_M}^{\clU_M^\perp}\widehat y(t)\notag
\end{align}
depending on the state estimate~$\widehat y(t)$ at time~$t$, provided by the Luenberger observer~\eqref{sys-num-haty}, with the forcing
\begin{align}
\clI(t)\coloneqq- \lambda_2 P_{\clW_S}^{\widetilde\clW_S^\perp}P_{\widetilde \clW_S}^{\clW_S^\perp}W_S^\diamond [\clV_W]^{-1} W_S^\vee(\widehat y-y),\notag
\end{align}
injected  into the dynamics of the observer, depending on the measured output~$W_S^\vee y$.

From the result established in Theorem \ref{thm:stab}, the solution of~\eqref{sys-num}
satisfies
\begin{align}
&\norm{(y,\widehat y-y)(t)}{H\times H} \le C \mathrm{e}^{-\mu (t-s)}\left(\norm{(y,\widehat y-y)(s)}{H\times H}+(D_{\rmm,s}^{y,\widehat y-y})^\frac12\right),\label{exp-dec-num}
\end{align}
for all~$t>s\ge0$, provided we choose~$M,S,\lambda_1(M),\lambda_2(S)$ large enough, with the constant~$D_{\rmm,s}^{y,\widehat y-y}$ depending on~$s$; see~\eqref{goal-exp-CLerror-Intro}.

\begin{remark} Recalling again~\eqref{output-Proj} and~\eqref{ProjOb=ObOr}, we can rewrite, shortly,
\begin{align}
U_M^\diamond u(t)&=- \lambda_1P_{\clU_M}^{\widetilde\clU_M^\perp}P_{\widetilde\clU_M}^{\clU_M^\perp}\widehat y(t);\qquad
\clI(t)=- \lambda_2 P_{\clW_S}^{\widetilde\clW_S^\perp}P_{\widetilde \clW_S}^{\clW_S^\perp}(\widehat y-y).\notag
\end{align}
\end{remark}

\subsection{Parameters}
For the simulations, the spatial domain is taken as the unit square and we consider the system subject to Neumann boundary conditions,
\begin{align}\label{num-OmegNeu}
\Omega= (0,1)\times(0,1)\subset\bbR^2;\qquad  (\bfn \cdot\nabla y)\rest{\p\Omega}=0= (\bfn \cdot\nabla \widehat y)\rest{\p\Omega}
\end{align}
Further, we have chosen
\begin{subequations}\label{num-param}
\begin{align}
&y_0 = 1-2x_1x_2;\quad&&\widehat{y}_0 = P_{\clW_S}^{\clW_S^\perp}y_0;\label{num-param-ic}\\
&a(x_1,x_2,t) =  -1.5+x_1-|(\sin(6t+x_1)|);\quad &&\nu=0.1; \\
&b(x_1,x_2,t) = (x_1 + x_2 ,|\cos(6t)x_1x_2|).\quad&&
\end{align}
\end{subequations}
For kernels as~\eqref{exk-intro} we shall take~$\gamma=1$, and for~\eqref{wsk-intro} we take~$\gamma=0.5$.  Finally, we are going to consider several values for the memory coefficient~$\eta$. 

In~\eqref{num-param-ic}, we propose to choose~$\widehat{y}_0 = P_{\clW_S}^{\clW_S^\perp}y_0$ as the guess for the initial state. Note that such~$\widehat{y}_0$ can be constructed from the available output~$w(0)=W_S^\vee y_0$ at time~$t=0$, because recalling~\eqref{output-Proj} we have that~$\widehat{y}_0=W_S^\diamond [\clV_W]^{-1}w(0)$.

By Lemmas~\ref{lem:actsens} and~\ref{lem:actsens-orth} we can take as auxiliary functions either ``regularized indicator functions'' as in~\cite[Sect.~6]{Rod21-aut} or simply the actuators and sensors themselves. For simplicity, we shall take the latter, that is, $\widetilde U_M=U_M$ and~$\widetilde W_S=W_S$; in other words, we shall take the feedback input and output injection operators as follows
\begin{align}\label{KL-num-1}
K_M^{[\lambda_1]}= - \lambda_1[\clV_U]^{-1}U_M^\vee P_{\clU_M}^{\clU_M^\perp},\qquad L_S^{[\lambda_2]}= - \lambda_2 W_S^\diamond [\clV_W]^{-1}.
\end{align}
We shall take the actuators and sensors as in Fig.~\ref{fig:suppActSensGrid}, corresponding to
\begin{equation}\label{MS-num}
(M,S)\in\{(1,1),(2,2),(3,3)\}.
\end{equation} 
In this way, the dynamics of the coupled system~\eqref{sys-num} becomes
\begin{subequations}\label{sys-num-2}
\begin{align}
&\tfrac{\partial}{\partial t}y=-Ay  -  a y - b \cdot\nabla y  -\eta \clA_\rmm[y] +U_M^\diamond K_M^{[\lambda_1]}\widehat y,\label{sys-num-y-2}\\
&\tfrac{\partial}{\partial t}{\widehat y}=-A\widehat y  -  a\widehat y - b \cdot\nabla \widehat y  - \eta\clA_\rmm[\widehat y]+U_M^\diamond K_M^{[\lambda_1]}\widehat y+L_S^{[\lambda_2]}(W_S^\vee\widehat y-W_S^\vee y),\label{sys-num-haty-2}\\
\mbox{with}&\quad K_M^{[\lambda_1]}= - \lambda_1[\clV_U]^{-1}U_M^\vee,\qquad L_S^{[\lambda_2]}= - \lambda_2 W_S^\diamond [\clV_W]^{-1},\label{sys-num-KL-2}\\
&\quad A=-\nu\Delta+\Id,\quad\mbox{and}\quad \clA_\rmm[h]\coloneqq\int_0^t\fkK(t-s)(-\Delta)h(s)ds,\label{sys-num-AAxi}
\end{align}
\end{subequations}
Note that~\eqref{KL-num-1} simplifies into~\eqref{sys-num-KL-2}. Here~$A=-\nu\Delta+\Id$ is the (shifted scaled) Neumann Laplacian with domain~$\rmD(A)=\{h\in W^{2,2}(\Omega)\mid (\bfn\cdot\nabla h)\rest{\partial\Omega}=0\}$; see~\eqref{num-OmegNeu}.

\begin{remark}
Though we know that the closed-loop system is stable, the performance of the output-based feedback-input will depend on the initial guess. If~$\widehat{y}_0$ is far from~${y}_0$, then the initial input~$K_M\widehat{y}_0$ is far from the desired~$K_M\widehat{y}_0$, and to guarantee that~$K_M\widehat{y}_0$ will have the desired decreasing effect on the norm of the controlled state~$y$, we need to wait until the estimate~$\widehat{y}(t)$ provided by the observer approaches~${y}(t)$ well enough. The choice~$\widehat{y}_0 = P_{\clW_S}^{\clW_S^\perp}y_0$, in~\eqref{num-param-ic}, explores the available information we have on the initial state~${y}_0$ trhough the output~$w(0)=W_S^\vee y_0$.
\end{remark}

\subsection{Discretization}
To compute the solution pair~$(y(t),\widehat y(t))$, for the spatial variable we consider a triangulation of the spatial domain and use a finite element discretization based on  piecewise linear basis functions (hat functions); for the temporal variable we use an implicit-explicit Crank--Nicolson/Adams--Bashforth scheme.

\subsubsection{Finite element spatial discretization}
We follow~\cite[Sect.~3.1]{Rod23-eect}. 
Let~$\bfS$ and~$\bfM$ be the stiffness and mass matrices and denote~$\bfS_{\nu}=\nu\bfS+\bfM$.
 Let~$\bfG_{x_i}$ be the discretizations of the directional derivatives~$\frac{\p}{\p x_i}$ and let~$\bfD_{\underline v}$ be the diagonal matrix, the entries of which
are those of the vector~$\underline v\in\bbR^{N\times 1}$, ~$(\bfD_{\underline v})_{(n,n)}=\underline v_{(n,1)}$.
Spatial discretization of~\eqref{sys-num-y-2} leads to
\begin{align}\label{sys-y-D1}
 &\bfM\dot{\underline y} =-\bfS_\nu \underline  y-\tfrac{\bfM \bfD_{\underline a}
 +\bfD_{\underline a}\bfM}{2}\underline y
 -{\textstyle\sum\limits_{i=1}^d}(\bfD_{\underline b_i}\bfG_{x_i})\underline  y
 -\eta\fkA_\rmm[\underline y]+\bfM\bfU\bfK\widehat{\underline y},
\end{align}
where~$\underline y(t)\in\bbR^{N\times 1}$ is the vector of values of the state at the spatial mesh (triangulation) points at time~$t\ge0$, and
\begin{equation}\label{bfU}
\bfU\coloneqq[\underline\indf_{\omega^1}\dots \underline\indf_{\omega^{M_\sigma}}]\in\bbR^{N\times M_\sigma}
\end{equation}
is the matrix the columns of which contain the finite-elements vectors corresponding to the actuators. Finally, $\bfu(t)=\bfK\underline y(t)\in\bbR^{M_\sigma\times 1}$ is the computed input feedback control~$u(t)=K_M^{[\lambda_1]} y(t)$, as follows
\begin{equation}\label{bfK}
\bfK\coloneqq -\lambda_1\bfV_U^{-1}\bfU^\top\bfM\quad\mbox{where}\quad\bfV_U\quad\mbox{has entries}\quad{\bfV_U}_{(i,j)}\coloneqq \underline\indf_{\omega^j}^\top\bfM\underline\indf_{\omega^i},
\end{equation}
in the $i$th row and~$j$th column; recall that~$[\clV_U]$, in~\eqref{KL-Oper}, has entries~$ (\indf_{\omega^j},\indf_{\omega^i})_{H}$.

Finally, for the spatial discretization~$\fkA_\rmm[\underline y]$ of the memory term (not considered in~\cite[Sect.~3.1]{Rod23-eect}) is taken as follows
\begin{equation}\label{fkAxi}
\fkA_\rmm[\underline y](t)=\int_0^t(t-s)^{\gamma-1}\bfS\underline y(s)\,\rmd s.
\end{equation}

Denoting, by simplicity
\begin{align}\notag
 \bfR(t)\!\coloneqq\tfrac{\bfM \bfD_{\underline a(t)}+\bfD_{\underline a(t)}\bfM}{2},
 \quad\!\! \bfC(t)\!\coloneqq\!
 {\textstyle\sum\limits_{i=1}^d}(\bfD_{\underline b_i(t)}\bfG_{x_i}),
 \quad\!\! \bfF^{y,\widehat y}(t)\!\coloneqq\!-\eta\fkA_\rmm[\underline y](t)+\bfM\bfU\bfK\widehat{\underline y}(t),
\end{align}
we arrive at the semi-discretization
\begin{align}\notag
 &\bfM\dot{\underline y} =-\bfS_\nu \underline  y-\bfR\underline y-\bfC\underline y
+\bfF^{y,\widehat y}.
\end{align}

\subsubsection{Crank--Nicolson/Adams--Bashforth temporal discretization}

We consider a regular temporal mesh, with
a fixed time-step~$k>0$, that is, we consider the discrete times instants~$t_j\coloneqq k(j-1)$, $j\in\bbN_+$.
Again, for simplicity, we denote
\begin{align}\notag
 \underline y_j\coloneqq  \underline y(t_j),\qquad \bfR_j\coloneqq \bfR(t_j),
 \qquad \bfC_j\coloneqq\bfC(t_j),
 \qquad \bfF^{y,\widehat y}_j\coloneqq\bfF^{y,\widehat y}(t_j).
\end{align}

Following~\cite[Sect.~3.1]{Rod23-eect}, we take an implicit Crank--Nicolson scheme for the sparse symmetric matrix~$-\bfS_\nu -\bfR$ combined with an explicit Adams--Bashforth extrapolation for the remaining~$-\bfC\underline y+\bfF_y$. This leads us to the implicit-explicit (IMEX) scheme as follows (cf.\cite[Eq.~(3.5)]{Rod23-eect}) 
\begin{align}
 &(2\bfM +k\bfS_\nu +k\bfR_{j+1}){\underline y}_{j+1}=h_j,\quad j\in\bbN_+,\quad\mbox{with}\label{sys-y-D2}\\
 &h_j\coloneqq(2\bfM -k\bfS_\nu -k\bfR_j){\underline y}_{j}
  -k(3\bfC_{j} {\underline y}_{j}-\bfC_{j-1} {\underline y}_{j-1})
  +k(3\bfF^{y,\widehat y}_{j}-\bfF^{y,\widehat y}_{j-1}),\notag
\end{align}
which we can solve to obtain~${\underline y}_{j+1}$, provided we know~$({\underline y}_{j-1},{\underline y}_{j})$.

Note that~${\underline y}_{1}=\underline{y_0}$, at initial
time~$t=t_1=0$, will be given by the initial state~$y_0$ in~\eqref{num-param-ic}. However,
to start the scheme, in order to obtain~${\underline y}_{2}$ at time~$t=k$, we need the
``ghost'' state~${\underline y}_{0}$ ``at ghost time~$t=t_0=-k$''. As in~\cite[Sect.~3.1]{Rod23-eect}, we have set/chosen~${\underline y}_{0}\coloneqq\underline{y_0}$.

Finally, we observe that
\begin{align}
\bfF^{y,\widehat y}_{j-1}=-\eta\fkA_\rmm[\underline y](t_{j-1})+\bfM\bfU\bfK\widehat{\underline y}_{j-1},
\end{align}
again for~$j=1$ we take the ghost~$\bfF^{y,\widehat y}_{0}\coloneqq\bfF^{y,\widehat y}_{1}=-\eta\fkA_\rmm[\underline y](0)+\bfM\bfU\bfK\widehat{\underline y}(0)=\bfM\bfU\bfK\underline{\widehat{y_0}}$.

Combining~\eqref{sys-y-D2} with the above ghosts we arrive at
\begin{align}
&\bfX_{j+1}^+{\underline y}_{j+1}= \bfX_{j}^-{\underline y}_{j}
-k\bfE^b_j-k\bfE^\rmm_j+k\widehat\bfE^U_j,\qquad  j\in\bbN_+,\label{sys-y-discfull0}\\
\mbox{where}\quad&{\underline y}_{0}={\underline y}_{1}\coloneqq\underline{y_0};\qquad\fkA_\rmm[\underline y](t_0)\coloneqq\fkA_\rmm[\underline y](t_1)=\fkA_\rmm[\underline y](0)=0;\notag\\
&\bfX_{j+1}^+\coloneqq2\bfM +k\bfS_\nu +k\bfR_{j+1};\quad\bfX_{j}^-
\coloneqq2\bfM -k\bfS_\nu -k\bfR_j;\notag\\
&\bfE^b_j\coloneqq3\bfC_{j} {\underline y}_{j}-\bfC_{j-1} {\underline y}_{j-1};\qquad
\bfE^\rmm_j\coloneqq3\eta\fkA_\rmm[\underline y](t_{j})-\eta\fkA_\rmm[\underline y](t_{j-1});\notag\\
&\widehat\bfE^U_j\coloneqq3\bfM\bfU\bfK\widehat{\underline y}_{j}-\bfM\bfU\bfK\widehat{\underline y}_{j-1} .\notag
\end{align}

Analogously, for the observer~\eqref{sys-num-KL-2} we take the discretization
\begin{align}
&\bfX_{j+1}^+{\widehat{\underline y}}_{j+1}= \bfX_{j}^-{\widehat{\underline y}}_{j}
  -k\widehat\bfE^b_j-k\widehat\bfE^\xi_j+k\widehat\bfE^U_j+k\widehat\bfE^W_j,\qquad  j\in\bbN_+,\label{sys-haty-discfull0}\\
\mbox{where}\quad
   &\widehat{\underline y}_{0}=\widehat{\underline y}_{1}\coloneqq\underline{\widehat y_0}\qquad\fkA_\rmm[\widehat{\underline y}](t_0)\coloneqq\fkA_\rmm[\widehat{\underline y}](t_1)=\fkA_\rmm[\widehat{\underline y}](0)=0;\notag\\
  &\widehat\bfE^b_j\coloneqq3\bfC_{j} \widehat{\underline y}_{j}-\bfC_{j-1} \widehat{\underline y}_{j-1};\qquad
  \widehat\bfE^\rmm_j\coloneqq3\eta\fkA_\rmm[\widehat{\underline y}](t_{j})-\eta\fkA_\rmm[\widehat{\underline y}](t_{j-1});\notag\\
  &\widehat\bfE^W_j\coloneqq 3\bfL(\bfW^\top\bfM\widehat{\underline y}_{j}-\bfW^\top\bfM{\underline y}_{j})-\bfL(\bfW^\top\bfM\widehat{\underline y}_{j-1}-\bfW^\top\bfM{\underline y}_{j-1}).\notag
\end{align}
with the output injection matrix
\begin{equation}\label{bfL}
\bfL\coloneqq -\lambda_2\bfM\bfW\bfV_W^{-1},\quad\mbox{where}\quad\bfV_W\quad\mbox{has entries}\quad{\bfV_W}_{(i,j)}\coloneqq \underline\indf_{\omega_{\rmm}^j}^\top\bfM\underline\indf_{\omega_{\rmm}^i},
\end{equation}
and~$\bfW$ is the matrix with columns given the the vectors corresponding to the sensors,
\begin{equation}\label{bfW}
\bfW\coloneqq[\underline\indf_{\omega_{\rmm}^1}\dots \underline\indf_{\omega_{\rmm}^{S_\varsigma}}]\in\bbR^{N\times S_\varsigma}.
\end{equation}

\subsubsection{The memory terms}
To obtain a fully discretized system, we specify the discretization of the memory terms~$\fkA_\rmm[\underline y](t_{j})$ and~$\fkA_\rmm[\widehat{\underline y}](t_{j})$, for~$j\ge 2$. We follow a variation of the discretization in~\cite[Sect.3.2]{MahajanKhanMohan24}. By~\eqref{fkAxi} and direct computations,
\begin{align}
\fkA_\rmm[\underline y](t_{j})&=\int_0^{t_j}\fkK(t_j-s)\bfS\underline y(s)\,\rmd s={\sum_{i=1}^{j-1}}\int_{t_{i}}^{t_{i+1}}\fkK(t_j-s)\bfS\underline y(s)\,\rmd s,\qquad j\ge2.\label{fkAxi-j-1}
\end{align}
Now, in the time interval~$[t_{i},t_{i+1}]$ we take the affine aproximation
\begin{align}
\underline y^*(s)&\coloneqq\underline y(t_{i})+k^{-1}(s-t_{i})({\underline y}(t_{i+1})-{\underline y}(t_{i})),\notag
\end{align}
satisfying~$\underline y^*(t_{i})=\underline y(t_{i})$ and~$\underline y^*(t_{i+1})=\underline y(t_{i+1})$.
It follows that
\begin{align}
&\int_{t_{i}}^{t_{i+1}}\!\!\fkK(t_j-s)\bfS\underline y(s)\,\rmd s\approx\int_{t_{i}}^{t_{i+1}}\!\!\fkK(t_j-s)\bfS\underline y^*(s)\,\rmd s\label{fkAxi-j-2}\\
&\hspace{-0em}=\int_{t_{i}}^{t_{i+1}}\!\!\fkK(t_j-s)\,\rmd s\bfS\underline y(t_{i})+k^{-1}\int_{t_{i}}^{t_{i+1}}\!\!(s-t_{i})\fkK(t_j-s)\,\rmd s\bfS({\underline y}(t_{i+1})-{\underline y}(t_{i})).\notag
\end{align}

Note that for the considered kernels, as in~\eqref{wsk-intro} and~\eqref{exk-intro}, we can compute the integrals~$\int_{t_{i}}^{t_{i+1}}\fkK(t_j-s)\,\rmd s$ and~$\int_{t_{i}}^{t_{i+1}}(s-t_{i}) {\fkK(t_j-s)}\,\rmd s$ exactly, as follows.

\noindent
$\bullet$\quad\underline{The case $\fkK(t)=\rme^{-\gamma t}$: We find}
\begin{align}
\clI_0^{\rm exk}&\coloneqq\int_{t_{i}}^{t_{i+1}}\rme^{-\gamma(t_j-s)}\,\rmd s =\left[\gamma^{-1}\rme^{-\gamma(t_j-s)}\right]_{t_{i}}^{t_{i+1}}=\gamma^{-1}\Bigl(\rme^{-\gamma(t_j-t_{i+1})}-\rme^{-\gamma(t_j-t_{i})}\Bigr)\notag\\
&=\gamma^{-1}\rme^{\gamma k(i-j)}\Bigl(\rme^{\gamma k}-1\Bigr),\label{int-o0sm}
\intertext{and}
\clI_1^{\rm exk}&\coloneqq\int_{t_{i}}^{t_{i+1}}(s-t_{i})\rme^{-\gamma(t_j-s)}\,\rmd s=\int_{0}^{k}r\rme^{-\gamma(t_j-r-t_{i})}\,\rmd r\notag\\
&=\left[\gamma^{-1}\rme^{-\gamma(t_j-t_{i}-r)}r\right]_{0}^{k}-\gamma^{-1}\int_{0}^{k}\rme^{-\gamma(t_j-t_{i}-r)}\,\rmd r\notag\\
&=\gamma^{-1}\rme^{-\gamma(t_j-t_{i}-k)}k+\left[-\gamma^{-2}\rme^{-\gamma(t_j-t_{i}-r)}\right]_{0}^{k}\notag
\intertext{and, by expanding the last sum,}
\clI_1^{\rm exk}&=\gamma^{-1}\rme^{\gamma k(1+i-j)}k-\gamma^{-2}\left(\rme^{-\gamma(t_j-t_{i}-k)}-\rme^{-\gamma(t_j-t_{i})}\right)\notag\\
&=\gamma^{-1}\rme^{\gamma k(1+i-j)}k-\gamma^{-2}\left(\rme^{\gamma k(1+i-j)}-\rme^{\gamma k(i-j)}\right).\label{int-o1sm}
\end{align}
Then, by~\eqref{fkAxi-j-1}, \eqref{fkAxi-j-2}, \eqref{int-o0sm}, and~\eqref{int-o1sm},  we obtain
\begin{subequations}\label{fkAxi-j.sm}
\begin{align}
&\fkA_\rmm[\underline y](t_{j})\approx\overline\fkA_\rmm[{\underline y}](t_{j})
\coloneqq\bfS{\sum_{i=1}^{j-1}}\Bigl(\Gamma_{1,j,i}^{\rm sm}\underline y(t_{i})+\Gamma_{2,j,i}^{\rm sm}({\underline y}(t_{i+1})-{\underline y}(t_{i}))\Bigr),\quad\mbox{with}\\
&\Gamma_{1,j,i}^{\rm sm}\coloneqq\gamma^{-1}\rme^{\gamma k(i-j)}\Bigl(\rme^{\gamma k}-1\Bigr);\\
&\Gamma_{2,j,i}^{\rm sm}\coloneqq\gamma^{-2}\rme^{\gamma k(i-j)}\Bigl(\gamma\rme^{\gamma k}-k^{-1}\left(\rme^{\gamma k}-1\right)\Bigr).\hspace{-.5em}
\end{align}
Analogously we take the discretization of the memory term in the observer as
\begin{align}
\fkA_\rmm[\widehat{\underline y}](t_{j})\approx\overline\fkA_\rmm[\widehat{\underline y}](t_{j})&\coloneqq \bfS{\sum_{i=1}^{j-1}}\Bigl(\Gamma_{1,j,i}^{\rm sm}\widehat{\underline y}(t_{i})+\Gamma_{2,j,i}^{\rm sm}(\widehat{\underline y}(t_{i+1})-\widehat{\underline y}(t_{i}))\Bigr).
\end{align}
\end{subequations}

\noindent
 {$\bullet$\quad\underline{The case $\fkK(t)=t^{\gamma-1}$: Now, we find}}
\begin{align}
\clI_0^{\rm wsk}&\coloneqq\!\int_{t_{i}}^{t_{i+1}}\!\!\!\!\!(t_j\!-\!s)^{\gamma-1}\,\rmd s\! =\! \left[-\gamma^{-1} (t_j\!-\!s)^{\gamma}\right]_{t_{i}}^{t_{i+1}}\!=\!\gamma^{-1}\Bigl((t_j\!-\!t_{i})^{\gamma}-(t_j\!-\!t_{i+1})^{\gamma}\Bigr)\notag\\
&=\gamma^{-1}k^{\gamma}\Bigl((j-i)^{\gamma}-(j-i-1)^{\gamma}\Bigr),\label{int-o0ws}\\
\mbox{and}\quad
\clI_1^{\rm wsk}&\coloneqq\int_{t_{i}}^{t_{i+1}}(s-t_{i})(t_j-s)^{\gamma-1}\,\rmd s=\int_{0}^{k}r(t_j-r-t_{i})^{\gamma-1}\,\rmd r\notag\\
&=\left[-\gamma^{-1}(t_j-t_{i}-r)^{\gamma}r\right]_{0}^{k}+\gamma^{-1}\int_{0}^{k}(t_j-t_{i}-r)^{\gamma}\,\rmd r\notag\\
&=-\gamma^{-1}(t_j-t_{i}-k)^{\gamma}k+\left[-\gamma^{-1}(\gamma+1)^{-1}(t_j-t_{i}-r)^{\gamma+1}\right]_{0}^{k}\notag
\intertext{and, by further direct computations,}
\clI_1^{\rm wsk}&=-\gamma^{-1}(t_j-t_{i}-k)^{\gamma}k\notag\\
&\quad-\gamma^{-1}(\gamma+1)^{-1}(t_j-t_{i}-k)^{\gamma+1}+\gamma^{-1}(\gamma+1)^{-1}(t_j-t_{i})^{\gamma+1}\notag\\
&=-\gamma^{-1}k^{\gamma+1}\Bigl((j-i-1)^{\gamma}+(\gamma+1)^{-1}(j-i-1)^{\gamma+1}\Bigr)\notag\\
&\quad+\gamma^{-1}(\gamma+1)^{-1}(j-i)^{\gamma+1}k^{\gamma+1}\notag\\
&=\gamma^{-1}(\gamma+1)^{-1}k^{\gamma+1}\Bigl((j-i)^{\gamma+1}-(j-i-1)^{\gamma+1}-(\gamma+1)(j-i-1)^{\gamma}\Bigr).\!\!\label{int-o1ws}
\end{align}
Then, by~\eqref{fkAxi-j-1}, \eqref{fkAxi-j-2}, \eqref{int-o0ws}, and~\eqref{int-o1ws},  we obtain
\begin{subequations}\label{fkAxi-j.ws}
\begin{align}
&\fkA_\rmm[\underline y](t_{j})\approx\overline\fkA_\rmm[{\underline y}](t_{j})
\coloneqq\bfS{\sum_{i=1}^{j-1}}\Bigl(\Gamma_{1,j,i}^{\rm ws}\underline y(t_{i})+\Gamma_{2,j,i}^{\rm ws}({\underline y}(t_{i+1})-{\underline y}(t_{i}))\Bigr),\quad\mbox{with}\\
&\Gamma_{1,j,i}^{\rm ws}\coloneqq\gamma^{-1}k^{\gamma}\Bigl((j-i)^{\gamma}-(j-i-1)^{\gamma}\Bigr);\\
&\Gamma_{2,j,i}^{\rm ws}\coloneqq\gamma^{-1}(\gamma+1)^{-1}k^{\gamma}\Bigl((j-i)^{\gamma+1}-(j-i-1)^{\gamma+1}-(\gamma+1)(j-i-1)^{\gamma}\Bigr).\hspace{-3em}
\end{align}
We take the analogue discretization for memory term in the observer, namely,
\begin{align}
\fkA_\rmm[\widehat{\underline y}](t_{j})\approx\overline\fkA_\rmm[\widehat{\underline y}](t_{j})&\coloneqq \bfS{\sum_{i=1}^{j-1}}\Bigl(\Gamma_{1,j,i}^{\rm ws}\widehat{\underline y}(t_{i})+\Gamma_{2,j,i}^{\rm ws}(\widehat{\underline y}(t_{i+1})-\widehat{\underline y}(t_{i}))\Bigr).
\end{align}
\end{subequations}

\subsubsection{The fully discrete system}
Recalling~\eqref{sys-y-discfull0} and~\eqref{sys-haty-discfull0},
by~\eqref{fkAxi-j.ws}/\eqref{fkAxi-j.sm}  the fully discretized system~\eqref{sys-num-2} reads,
\begin{align}\label{sys-num-fullydiscrete}
&\left\{\begin{array}{l}\bfX_{j+1}^+{\underline y}_{j+1}= \bfX_{j}^-{\underline y}_{j}
  -k\bfE^b_j-k\overline\bfE^\rmm_j+k\widehat\bfE^U_j,\\
  \bfX_{j+1}^+{\widehat{\underline y}}_{j+1}= \bfX_{j}^-{\widehat{\underline y}}_{j}
  -k\widehat\bfE^b_j-k\overline{\widehat\bfE}^\rmm_j+k\widehat\bfE^U_j+k\widehat\bfE^W_j,
  \end{array}\right.\qquad  j\in\bbN_+,
  \intertext{where}
    &\bfX_{j+1}^+\coloneqq2\bfM +k\bfS_\nu +k\bfR_{j+1};\quad\bfX_{j}^-
\coloneqq2\bfM -k\bfS_\nu -k\bfR_j;\notag\\
  &{\underline y}_{0}={\underline y}_{1}\coloneqq\underline{y_0};\qquad\widehat{\underline y}_{0}=\widehat{\underline y}_{1}\coloneqq\underline{\widehat y_0};\notag\\
  &\bfE^b_j\coloneqq3\bfC_{j} {\underline y}_{j}-\bfC_{j-1} {\underline y}_{j-1};\qquad\widehat\bfE^U_j\coloneqq3\bfM\bfU\bfK\widehat{\underline y}_{j}-\bfM\bfU\bfK\widehat{\underline y}_{j-1};\notag\\
  &\widehat\bfE^W_j\coloneqq 3\bfL(\bfW^\top\bfM\widehat{\underline y}_{j}-\bfW^\top\bfM{\underline y}_{j})-\bfL(\bfW^\top\bfM\widehat{\underline y}_{j-1}-\bfW^\top\bfM{\underline y}_{j-1});\notag\\
  &\overline\bfE^\rmm_j\coloneqq3\eta\overline\fkA_\rmm[\underline y](t_{j})-\eta\overline\fkA_\rmm[\underline y](t_{j-1});\quad \overline\fkA_\rmm[\underline y](t_0)\coloneqq\overline\fkA_\rmm[\underline y](t_1)=\overline\fkA_\rmm[\underline y](0)=0;\notag\\
  &\overline{\widehat\bfE}^\rmm_j\coloneqq3\eta\overline\fkA_\rmm[\widehat{\underline y}](t_{j})-\eta\overline\fkA_\rmm[\widehat{\underline y}](t_{j-1});\quad \overline\fkA_\rmm[\widehat{\underline y}](t_0)\coloneqq\overline\fkA_\rmm[\widehat{\underline y}](t_1)=\overline\fkA_\rmm[\widehat{\underline y}](0)=0;\notag
\end{align}
with~$\overline\fkA_\rmm[{\underline y}](t_{j})$ and~$\overline\fkA_\rmm[\widehat{\underline y}](t_{j})$ as in~\eqref{fkAxi-j.ws}/\eqref{fkAxi-j.sm}, $\bfK$ as in~\eqref{bfK}, and $\bfL$ as in~\eqref{bfL}.

\subsection{Spatio-temporal meshes}
\begin{figure}[htbp]
\centering
\subfigure[Spatial triangulation~$\fkT_0^{2\times2}$.\label{fig:meshesfkM0-2}]
{\includegraphics[width=0.310\textwidth]{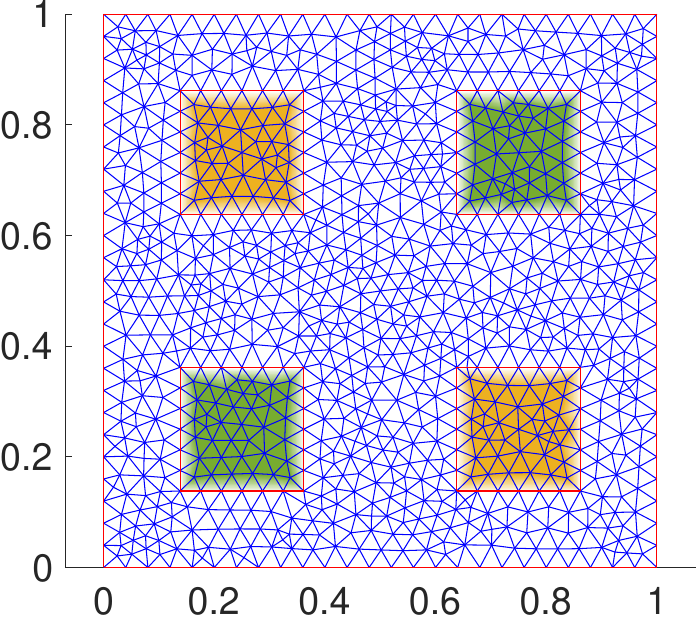}}
\;
\subfigure[Spatial triangulation~$\fkT_0^{4\times4}$.\label{fig:meshesfkM0-4}]
{\includegraphics[width=0.310\textwidth]{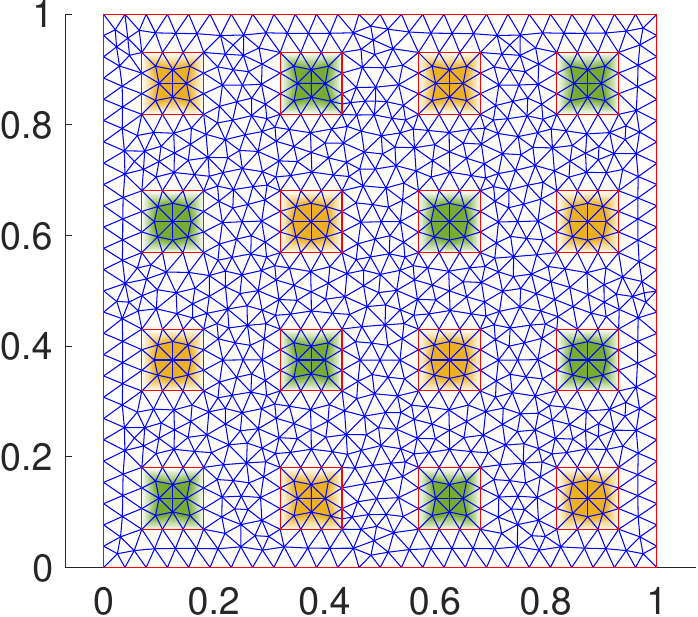}}
\;
\subfigure[Spatial triangulation~$\fkT_0^{6\times6}$.\label{fig:meshesfkM0-6}]
{\includegraphics[width=0.310\textwidth]{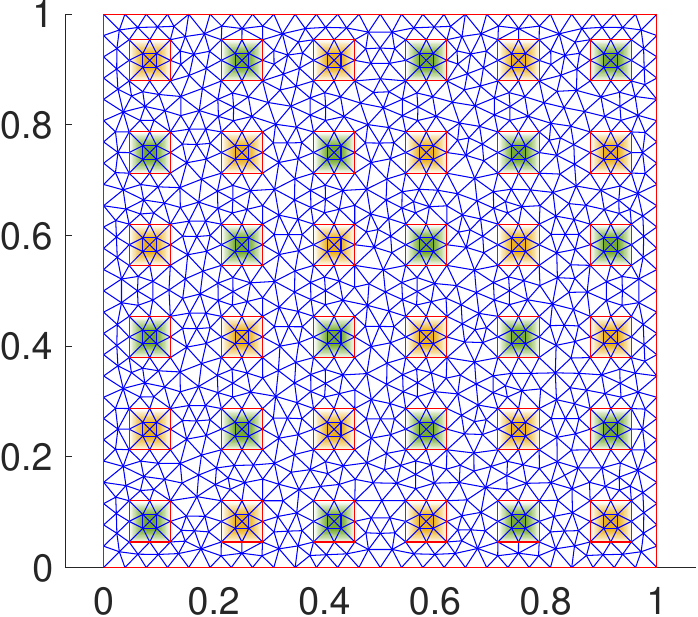}}
\caption{Coarsest spatial triangulations~$\fkT=\fkT_0^{\ell\times\ell}$, $\ell\in\{2,4,6\}$. }\label{fig:meshesfkM0}
\end{figure}
In Fig.~\ref{fig:meshesfkM0} we see triangulations of the spatial domain depending on the supports of the actuators and sensors in a chessboard pattern as is Fig.~\ref{fig:suppActSensGrid}; with one of the actuators placed at the left-botton corner.
Concerning the time-step we have chosen~$k=t^{\rm step}_0=4\cdot 10^{-4}$. This gives us a pair~$(k,\fkT)=(t^{\rm step}_0,\fkT_0^{\ell\times\ell})\eqqcolon \fkM_0^{\ell\times\ell}$ defining the  spatio-temporal discretization. 

Some simulations, will also be performed on refinements~$\fkM_{\rm rf}^{\ell\times\ell}$ of~$\fkM_0^{\ell\times\ell}$, ${\rm rf}\in\{1,2,3\}$. That is, we shall consider the three spatio-temporal meshes
\begin{equation}\label{meshes}
\fkM_{\rm rf}^{\ell\times\ell}\coloneqq(2^{-{\rm rf}} t^{\rm step}_0,\fkT_{\rm rf}^{\ell\times\ell}),\qquad 0\le {\rm rf}\le 3,\qquad t^{\rm step}_0=4\cdot 10^{-4},
\end{equation}
where~$\fkT_{\rm rf}^{\ell\times\ell}$ is a regular refinement of~$\fkT_{{\rm rf}-1}^{\ell\times\ell}$. 

\noindent
{\bf Notation:}
The data~``${\rm rf}$'' and~``$N=\ell\times\ell$'', in the annotation within the figures hereafter,  means that the simulations were performed in the spatio-temporal mesh~$\fkM_{{\rm rf}}^{\ell\times\ell}$. The function~$\ln=\log_\rme$ in the title of some figures is the natural (base~$\rme$) logarithm.

\subsection{Instability of the free dynamics}\label{ssec:num-free22}
In Fig.~\ref{fig:num-free-y} we can see that, with vanishing control forcing (i.e., with~$\lambda_1=0$), the norm of the state does not converge to zero. This means that the free dynamics of the equation~\eqref{sys-num-y-2} is unstable.
Such instability is also observed in Fig.~\ref{fig:num-free-dif}
\begin{figure}[htbp]
\centering
\subfigure[No output-injection\label{fig:num-free-dif}]
{\includegraphics[width=0.310\textwidth]{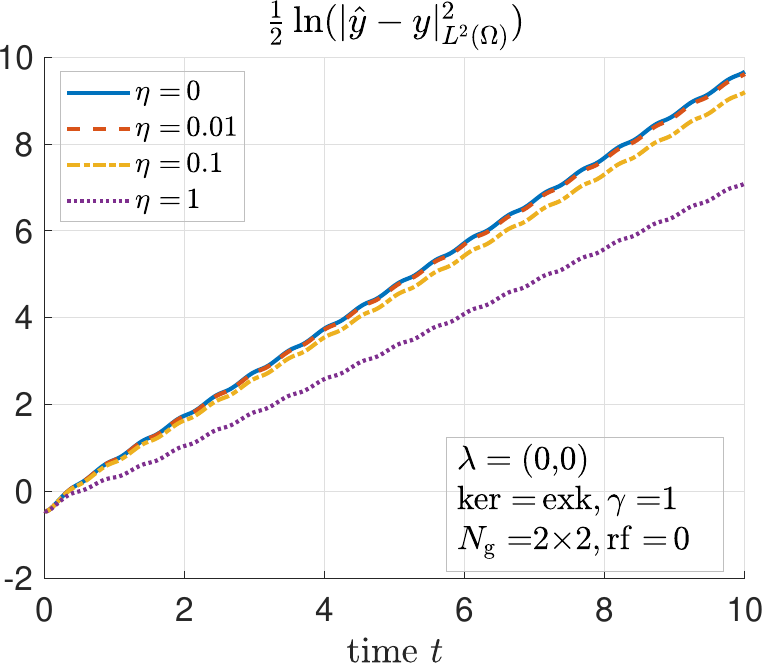}}
\qquad
\subfigure[No feedback-input\label{fig:num-free-y}]
{\includegraphics[width=0.310\textwidth]{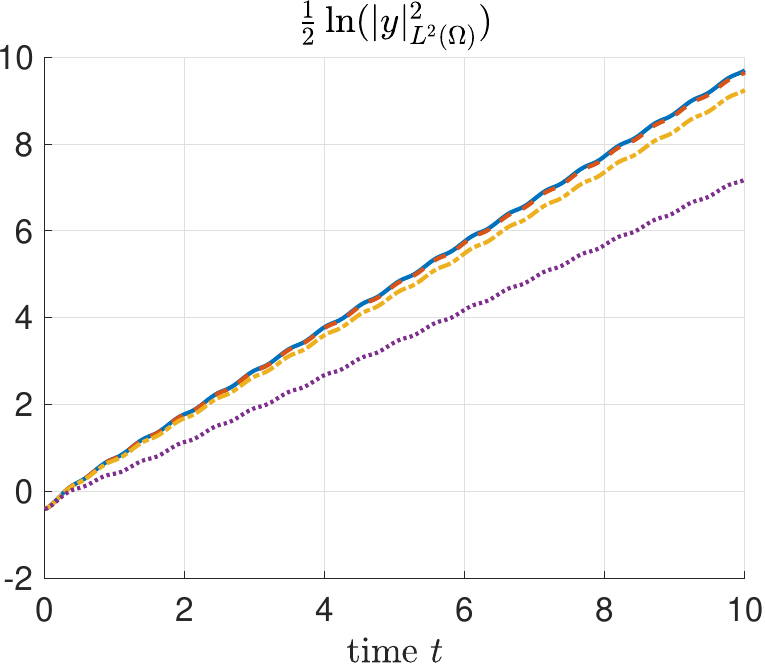}}
\caption{Closed-loop free dynamics~\eqref{sys-num-2}, for several values of~$\eta$.}\label{fig:num-free}
\end{figure}
where we see that, with vanishing output-injection forcing (i.e., with~$\lambda_2=0$) and with the same vanishing control forcing, the state estimate~$\widehat y(t)$ does not converge to~$y(t)$.
Therefore, we need a nonzero control forcing to stabilize the equation~\eqref{sys-num-y-2} and a nonzero output-injection forcing as well to guarantee that the estimate provided by the observer~\eqref{sys-num-haty-2}  approaches the (unknown) controlled state as time increases.

\begin{remark}\label{rem:mem-stab-effect}
Though the free dynamics is unstable for all~$\eta\in\{0,0.01,0.1,1\}$, Fig.~\ref{fig:num-free-y} shows that, in this experiment, the memory term reduces the instability of~\eqref{sys-num-y-2}, that is, the increase rate (i.e., ``the'' slope of the plotted lines) is smaller for larger~$\eta$, suggesting that the memory term has a damping effect in the dynamics.
\end{remark}

\subsection{Using  2 actuators and 2 sensors}\label{ssec:num22}
Here we fix~$\eta=0.01$ and consider actuators and sensors placed as in Fig.~\ref{fig:meshesfkM0-2}.  We show the results corresponding to $\lambda_1=200$ and several choices of~$\lambda_2$ in Fig.~\ref{fig:num-22eta.01},
\begin{figure}[htbp]
\centering
\subfigure
{\includegraphics[width=0.310\textwidth]{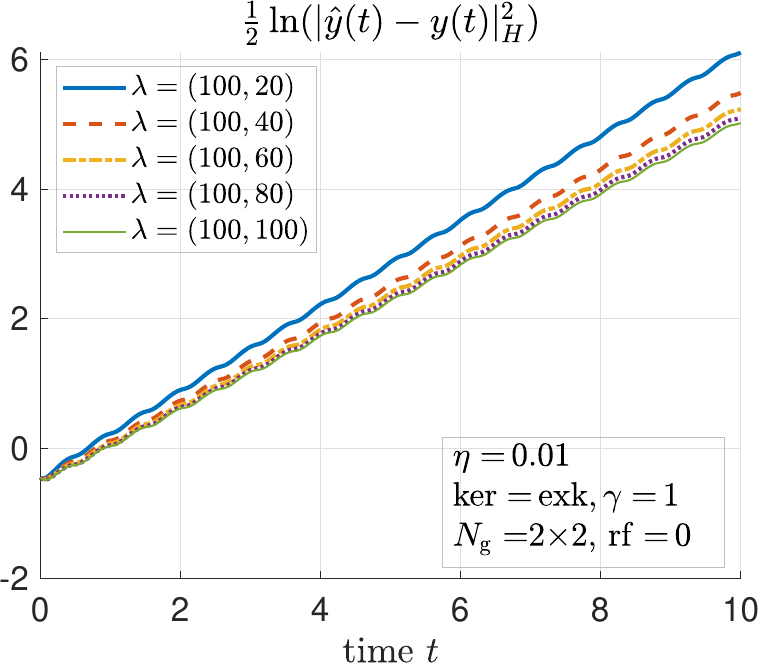}}
\qquad
\subfigure
{\includegraphics[width=0.310\textwidth]{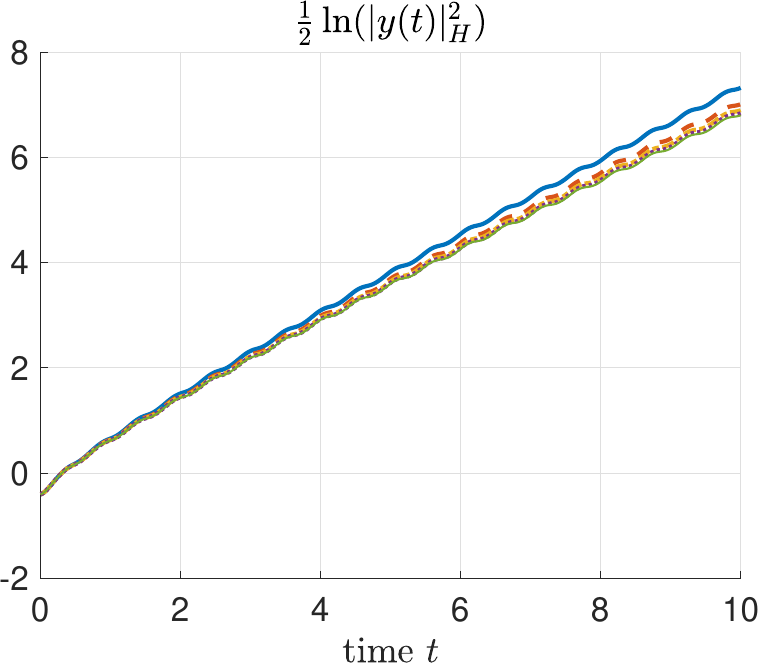}}
\caption{Using 2 actuators and 2 sensors as in~\eqref{sys-num-2}, for several choices of~$\lambda$.}\label{fig:num-22eta.01}
\end{figure}
where we see that the norms increase in all cases, for~$\eta=0.01$. We see also a converging behavior of the norms of the evolution of the norms of the solutions to an exponentially increasing function as~$\lambda_2$ increases. This observations, for example, suggest that increasing~$\lambda_2$ further will not be sufficient to achieve the desired convergence of the state estimate~$\widehat y(t)$ to the controlled state~$y(t)$,  as time~$t$ increases. 
Thus, we conclude that these 2 actuators and 2 sensors are not sufficient to guarantee the stability of the closed-loop coupled system~\eqref{sys-num-2} in the case~$\eta=0.01$. 
Next, in Fig.~\ref{fig:num-22eta0.11},
\begin{figure}[htbp]
\centering
\subfigure
{\includegraphics[width=0.310\textwidth]{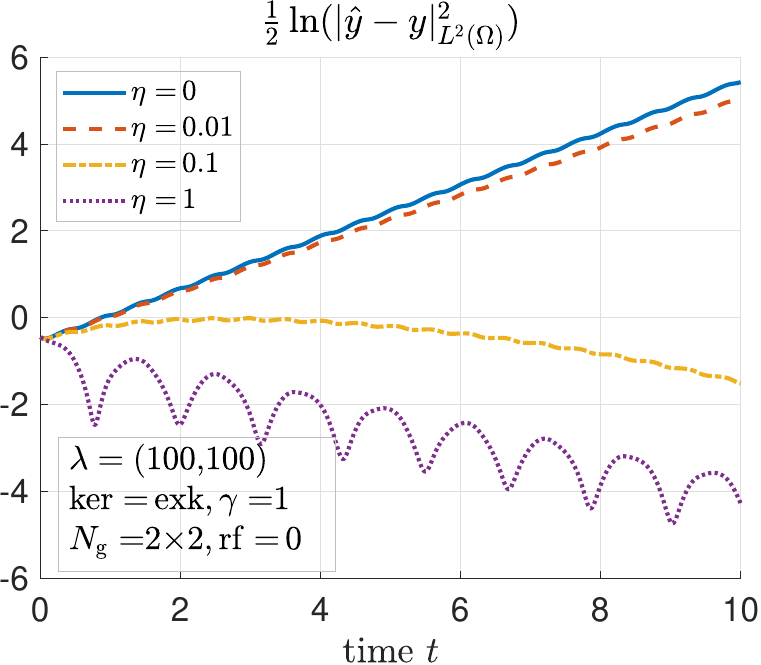}}
\qquad
\subfigure
{\includegraphics[width=0.310\textwidth]{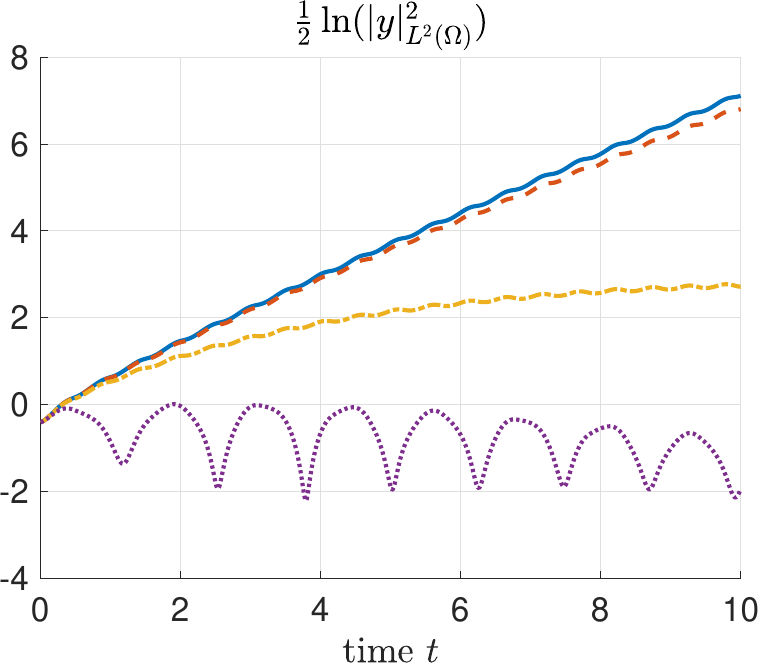}}
\caption{Using 2 actuators and 2 sensors as in~\eqref{sys-num-2}, for several values of~$\eta$.}\label{fig:num-22eta0.11}
\end{figure}
 we compare the results depending on the memory coefficient for a fixed~$\lambda=(\lambda_1,\lambda_2)$. We see that for smaller~$\eta\in\{0,0.01\}$ the considered 2 actuators and 2 sensors are not able to achieve the desired decrease of~$\norm{\widehat y(t)-y(t)}{H}$ and~$\norm{y(t)}{H}$. Instead, for larger~$\eta\in\{0.1,1\}$, the figure suggests again a damping effect of the memory term, in this case strong enough to indicate that the norm~$\norm{\widehat y(t)-y(t)}{H}$ may even decrease asymptotically to zero.

\begin{remark}\label{rem:stab-largeeta}
Within the proof of Theorem~\ref{thm:stab}, in the case of exponential kernels, the tuple~$(M,S,\lambda_1,\lambda_2)$ of parameters for feedback-input and output-injection  operators is chosen, essentially, so that the memoryless dynamics, in case~$\eta=0$, is stable, namely, so that the memoryless error dynamics is strictly decreasing. The stability of the controlled memoryless dynamics is not seen in Fig~\ref{fig:num-22eta0.11}, hence the stability observed in Fig~\ref{fig:num-22eta0.11}, for~$\eta\in\{0.1,1\}$, does not follow from the arguments within the proof.
\end{remark}

\subsection{Using 8 actuators and 8 sensors}\label{ssec:num44}
As shown in Fig.~\ref{fig:num-22eta.01} we know that~2 actuators and~2 sensors (as in Fig.~\ref{fig:meshesfkM0-2})
are not sufficient to obtain the desired exponential decrease for the norm~$\norm{\widehat y(t)-y(t)}{H}$ and for the norm~$\norm{y(t)}{H}$, in the case~$\eta=0.01$.

By the theoretical result that we presented as Theorem~\ref{thm:stab}, we know that a large enough number of sensors and actuators will provide us with the exponential decrease for~$\norm{z(t)}{H}^2+\norm{y(t)}{H}^2$, with~$z=\widehat y-y$, thus also for~$\norm{\widehat y(t)-y(t)}{H}$ and~$\norm{y(t)}{H}$ separately.

Therefore, we increase the number of actuators and sensors, by considering 
8 actuators and 8 sensors, placed as in Fig.~\ref{fig:meshesfkM0-4}, and take the same value of~$\lambda$ as in Fig.~\ref{fig:num-22eta0.11}. The results are reported in Fig.~\ref{fig:num-44eta0.11},
\begin{figure}[htbp]
\centering
\subfigure
{\includegraphics[width=0.310\textwidth]{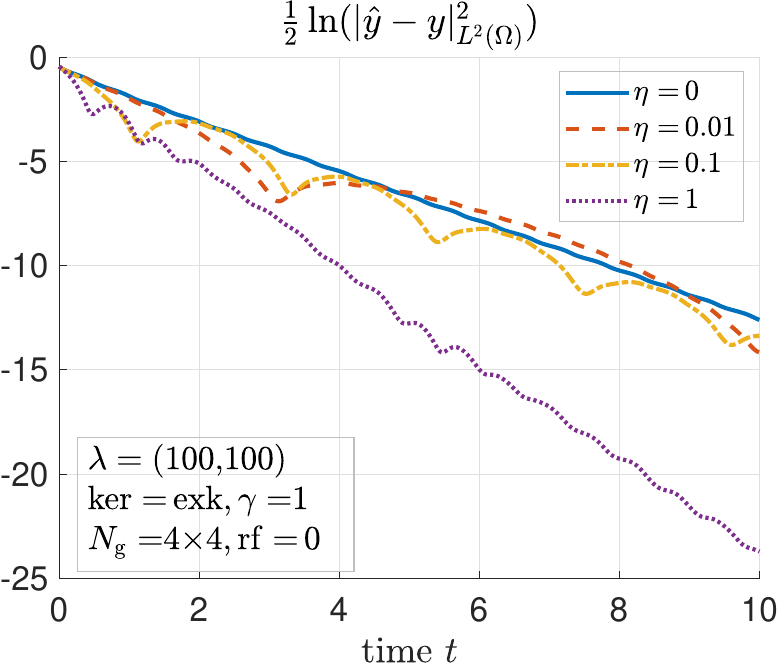}}
\qquad
\subfigure
{\includegraphics[width=0.310\textwidth]{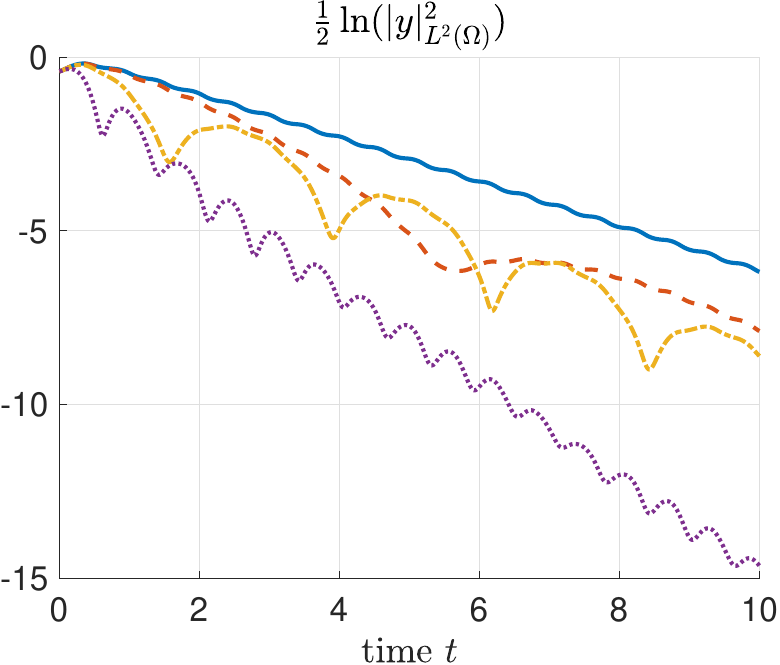}}
\caption{Using 8 actuators and 8 sensors as in~\eqref{sys-num-2}, for several values of~$\eta$.}\label{fig:num-44eta0.11}
\end{figure}
where we see that such actuators and sensors are able to provide us with the wanted exponential decrease for both~$\norm{\widehat y(t)-y(t)}{H}$ and~$\norm{y(t)}{H}$, in the cases~$\eta\in\{0,0.01,0.1,1\}$. Note that the smallest observed exponential decrease rate is approximately equal to~$1=\gamma$ (for the cases~$\eta<1$), which agrees with the estimate~\eqref{est-fkM-exp} we found in Theorem~\ref{thm:stab}, for large enough~$M$ and~$S$. Recall also that for the smaller~$\eta\in\{0,0.01\}$ (with the same~$\lambda$) the 2 actuators and 2 sensors as in Fig.~\ref{fig:meshesfkM0-4} are not providing the wanted decrease of the norms, as we have seen in Fig.~\ref{fig:num-22eta0.11}.
In particular, we see that in the memoryless case, the norm~$\norm{z(t)}{H}=\norm{\widehat y(t)-y(t)}{H}$ is strictly decreasing, which agrees with~\eqref{est-error-exp} and with the integral version as~\eqref{est-z-exp}, where the memory contributions do vanish,~$\fkM_z(s)$ for all~$s\ge0$, (cf.~\cite[Thm.~2.1]{Becker07}). On the other side in case~$\eta>0$ we see that~$\norm{z(t)}{H}$ is not strictly decreasing, which again agrees with the fact that constants~$\fkM_z(s)$ may be strictly positive, for some~$s>0$. In particular, we see that memory term has a relevant effect on the evolution of the norms~$\norm{z(t)}{H}$ and~$\norm{y(t)}{H}$.

\subsection{Using 18 actuators and 18 sensors}\label{ssec:num66}
We increase the number of actuators and sensors to~18, as in Fig.~\ref{fig:meshesfkM0-6}. In Fig.~\ref{fig:num-66eta0.11}
\begin{figure}[htbp]
\centering
\subfigure
{\includegraphics[width=0.310\textwidth]{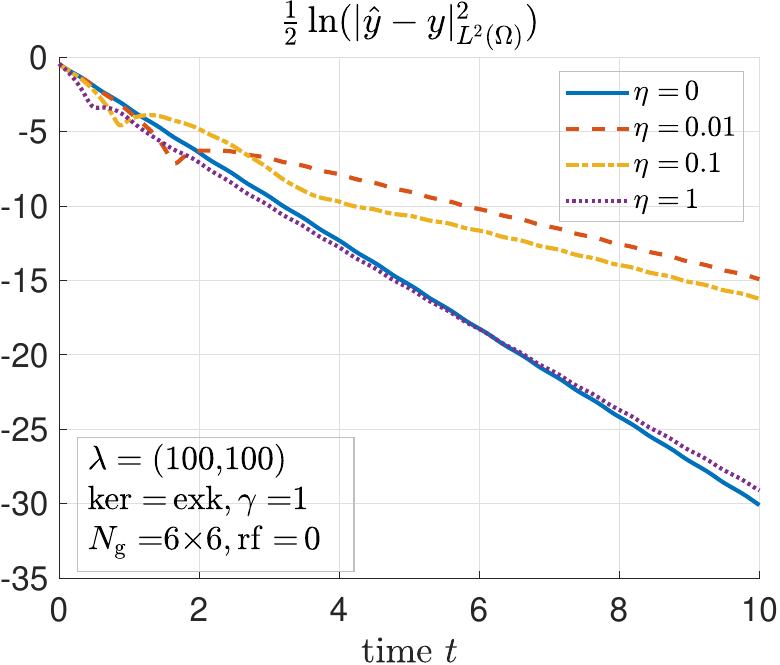}}
\qquad
\subfigure
{\includegraphics[width=0.310\textwidth]{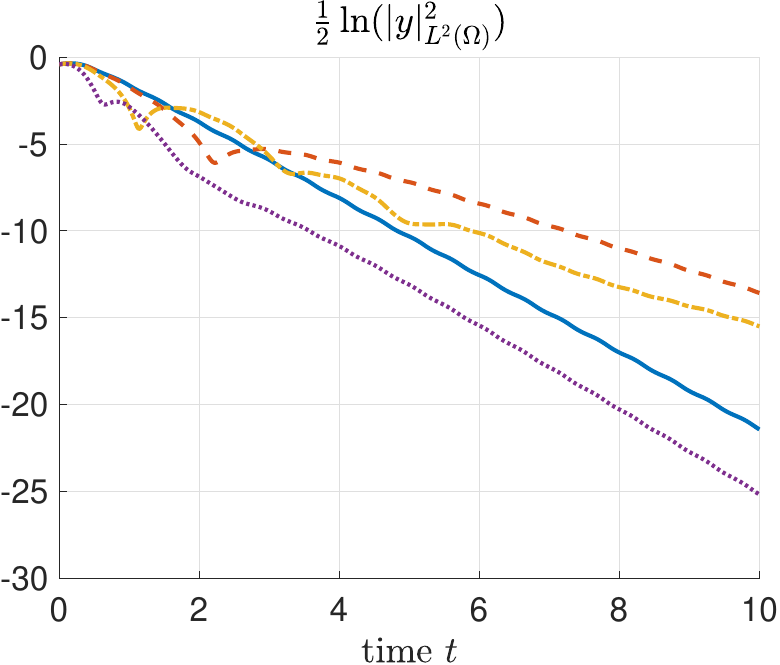}}
\caption{Using 18 actuators and 18 sensors as in~\eqref{sys-num-2}, for several values of~$\eta$.}\label{fig:num-66eta0.11}
\end{figure}
 we see again that the exponential decrease is achieved with rate~$1=\gamma$ validating again the result in Theorem~\ref{thm:stab}, for exponential kernels. As expected, by taking more actuators and sensors we obtain a quicker stabilization rate (approx.~$3$) for the memoryless case, however under the  presence of a memory term we obtain, for the cases~$\eta\in\{0.01,0.1\}$, only the exponential stabilization rate~$\gamma=1$ given by the kernel. This suggests that the kernel rate~$\gamma=1$ is the largest possible guaranteed stabilization rate, and consequently the statement in Theorem~\ref{thm:stab} is sharp, for exponential kernels and for the explicit feedback-input and output-injection operators that we propose.

\subsection{Remarks}\label{sec:num-exp-rmk}
The results of simulations, reported throughout Sections~\ref{ssec:num-free22}--\ref{ssec:num66}, validate the statement in Theorem~\ref{thm:stab}, for exponential kernels. In particular, an exponential decrease rate as~$\gamma$, will be achieved for~$\norm{\widehat y(t)-y(t)}{H}$ and~$\norm{y(t)}{H}$, provided the numbers of actuators and sensors, $M_\sigma$ and~$S_\varsigma$, and the values of feedback and injection scalar gains, $\lambda_1$ and~$\lambda_2$, are all large enough. Further, the results of simulations show that it may be not possible, in general, to achieve a stabilization rate larger
than the decrease rate~$\gamma$ of the exponential kernel.

\section{Additional results of simulations and comments}\label{sec:num-sim-res}

\subsection{Simulations with a weakly singular kernel}\label{sec:num-sim-wsk}
Figures~\ref{fig:num-22eta0.11wsk}--\ref{fig:num-66eta0.11wsk}
\begin{figure}[htbp]
\centering
\subfigure
{\includegraphics[width=0.310\textwidth]{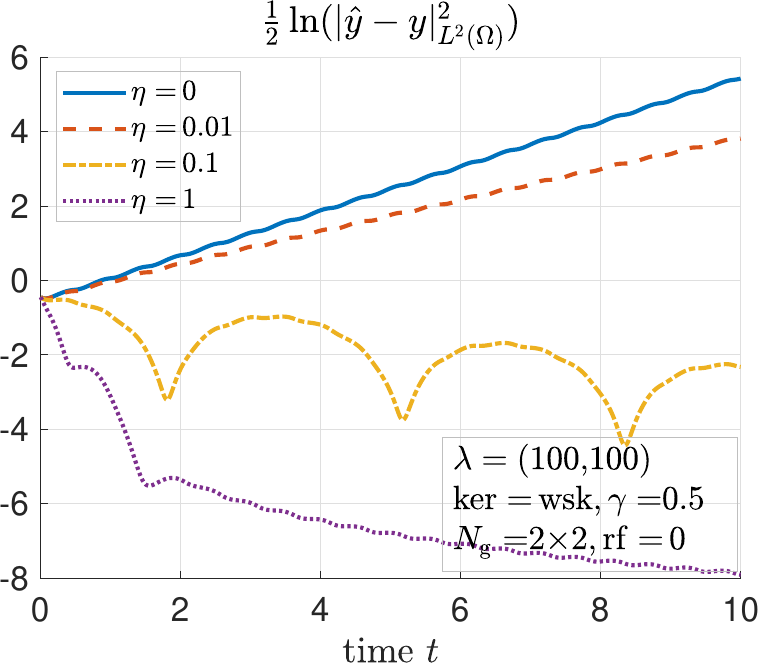}}
\qquad
\subfigure
{\includegraphics[width=0.310\textwidth]{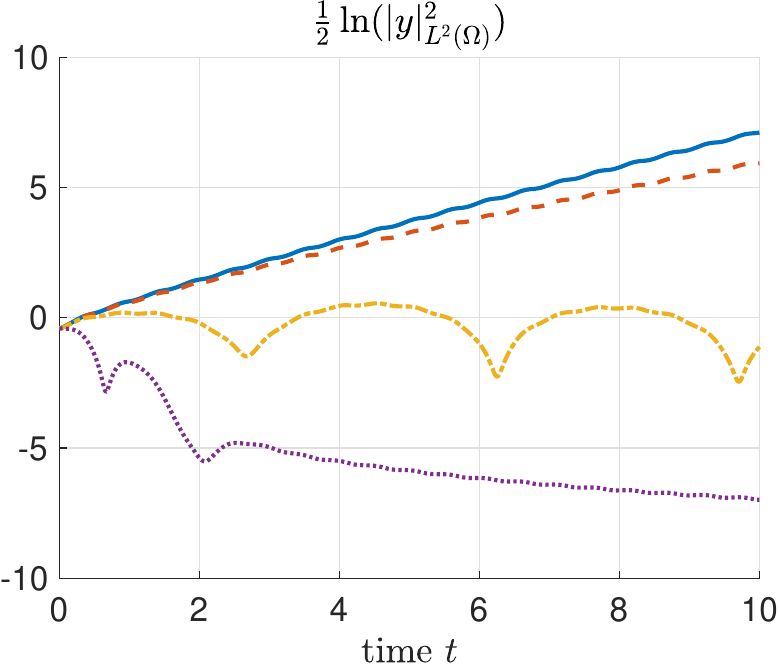}}
\caption{Using 4 actuators and 4 sensors in~\eqref{sys-num-2}, for several values of~$\eta$.}\label{fig:num-22eta0.11wsk}
\end{figure}
\begin{figure}[htbp]
\centering
\subfigure
{\includegraphics[width=0.310\textwidth]{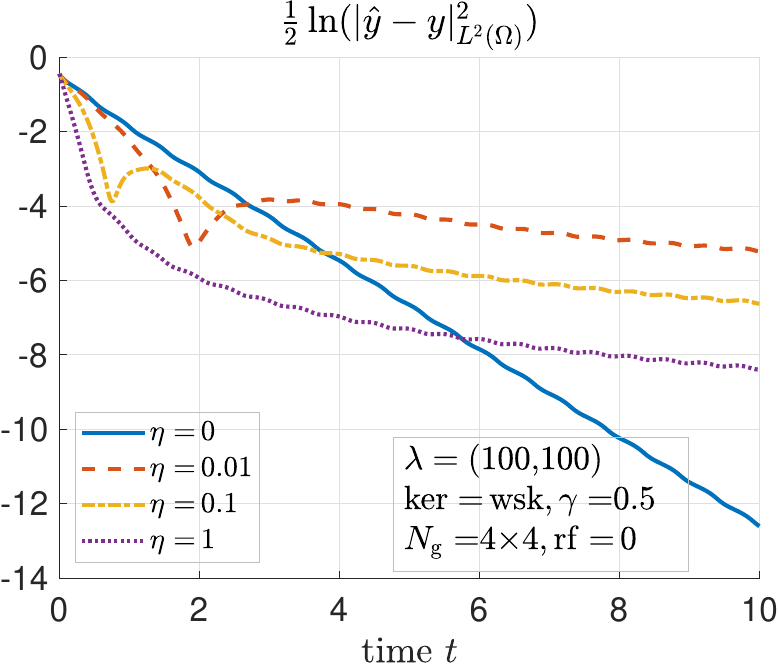}}
\qquad
\subfigure
{\includegraphics[width=0.310\textwidth]{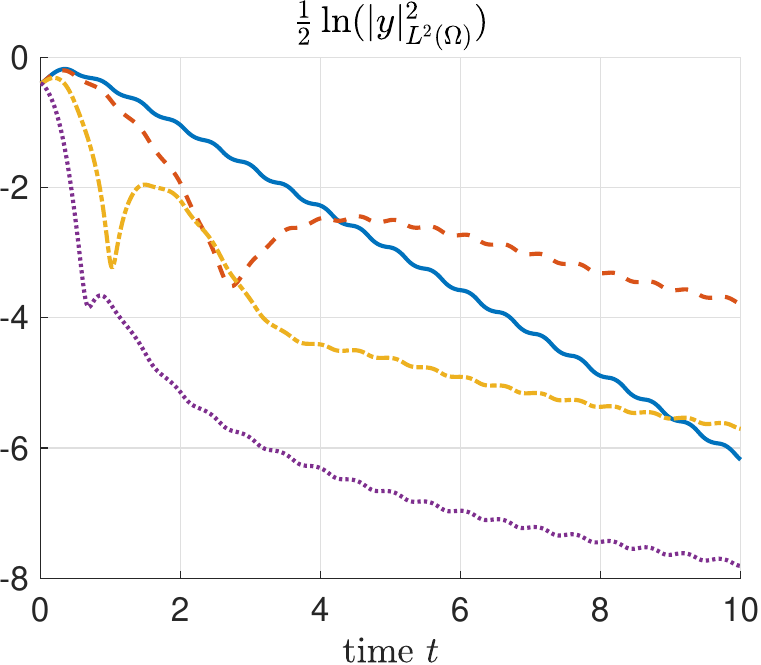}}
\caption{Using 8 actuators and 8 sensors as in~\eqref{sys-num-2}, for several values of~$\eta$.}\label{fig:num-44eta0.11wsk}
\end{figure}
\begin{figure}[htbp]
\centering
\subfigure
{\includegraphics[width=0.310\textwidth]{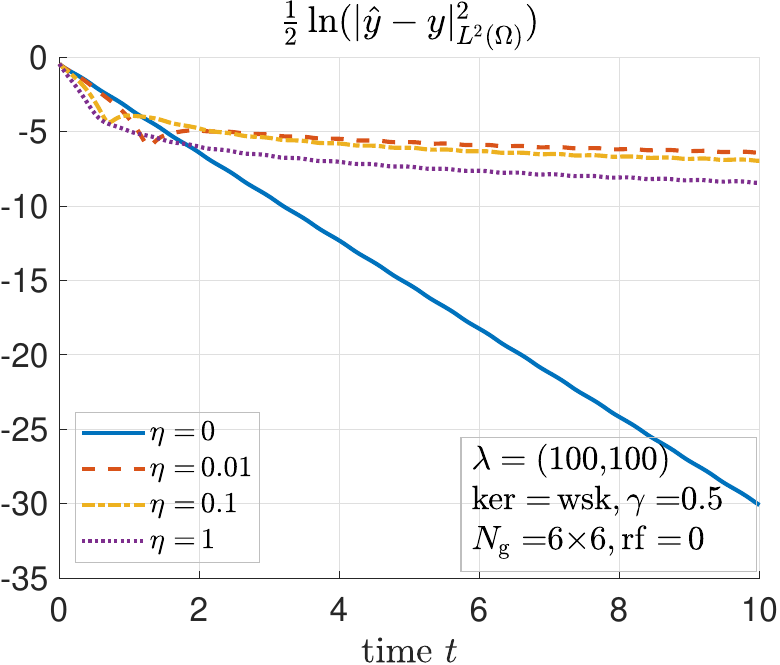}}
\qquad
\subfigure
{\includegraphics[width=0.310\textwidth]{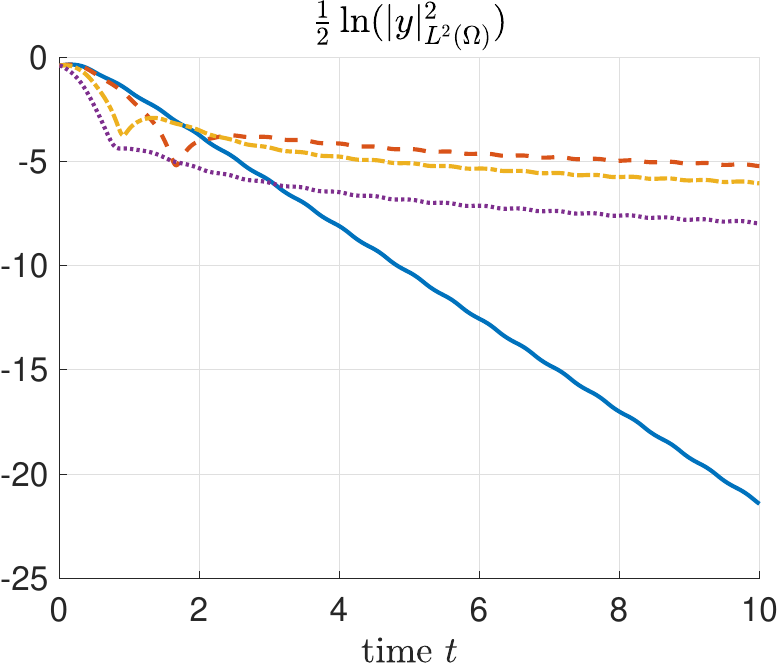}}
\caption{Using 18 actuators and 18 sensors as in~\eqref{sys-num-2}, for several values of~$\eta$.}\label{fig:num-66eta0.11wsk}
\end{figure}
show again the evolution of the norms of the state-estimate error~$\widehat y-y$ and of the controlled state~$y$. Now, corresponding to the case where the kernel is the weakly singular Riesz kernel~\eqref{wsk-intro}. The figures show that 
asymptotic stability can be achieved for large enough numbers of actuators and sensors, by looking at the negative slope of the plotted lines for large time. The likely asymptotic convergence to zero, observed in Figs.~\ref{fig:num-44eta0.11wsk}--\ref{fig:num-66eta0.11wsk}, is slower that that observed in Figs.~\ref{fig:num-44eta0.11}--\ref{fig:num-66eta0.11}. The figures also show that there exists a ``rate'' of asymptotic convergence of the norms to~$0$ which is not improved by taking more sensors and actuators. Thus, the results are qualitatively comparable to the case of exponential kernels. However, recall that for general kernels (other than exponential ones) we have only shown Lyapunov stability of the closed-loop system; see~\eqref{est-fkM-lyap}. The proof of the asymptotic stabilizability suggested in  Figs.~\ref{fig:num-22eta0.11wsk}--\ref{fig:num-66eta0.11wsk} is an interesting question.
At this point, we  mention a few works containing arguments that could be helpful towards a proof of asymptotic stabilizability, but also showing that this is a nontrivial problem. In~\cite[Thm.~(2.1)]{MacCamyWong72}, weak asymptotic stability is obtained in case we know that the solution is weakly bounded and weakly uniformly continuous; in~\cite[Thm.~4]{NohelShea76} a frequency domain criterion is used to derive the asymptotic stability; in~\cite[Thm.~3.4]{MaininiMola09} polynomial asymptotic stability results are derived, for the heat equation with memory, for smooth polynomial kernels~$\fkK(t)=(t+1)^{-p}$, $p>2$.

\subsection{Remarks on the numerical solution for refined meshes}
Recall the spatio-temporal meshes~$\fkM_{\rm rf}^{\ell\times\ell}$ in~\eqref{meshes}, ${\rm rf}\in\{0,1,2,3\}$.
We have run the previous simulations in the meshes~$\fkM_0^{\ell\times\ell}$. Here, we compare numerical solutions computed on these meshes to the corresponding ones computed in the refinements~$\fkM_{\rm rf}^{\ell\times\ell}$, ${\rm rf}\in\{1,2,3\}$. We investigate whether~$\fkM_0^{\ell\times\ell}$ is fine enough to capture essential qualitative properties of the solutions.

\subsubsection{Free dynamics: comparison to an exact solution}
We consider the function
\begin{equation}\label{yexa}
y^{\rm exa}(t,x_1,x_2)=y^{\rm exa}(x_1,x_2)\coloneqq\cos(\pi x_1)\cos(2\pi x_2)+\cos(2\pi x_1)+2.
\end{equation}
This function satisfies the considered Neumann boundary conditions and is the exact solution of the free dynamics~\eqref{sys-num-y-2}, provided we take the matching external forcing as follows
\begin{equation}\notag%
f^{\rm exa}\coloneqq \tfrac{\partial}{\partial t} y^{\rm exa}-\nu\Delta y^{\rm exa}+y_{\rm exa} +ay_{\rm exa}+b\cdot\nabla y_{\rm exa}+\eta\int_0^t(t-s)^{\gamma-1}(-\Delta)y_{\rm exa}(s)\,\rmd s.
\end{equation}
Note that, we can compute $f_{\rm exa}$ explicitly, namely, by straightforward computations, 
\begin{subequations}\label{fexa}
\begin{align}
f_{\rm exa}(t,x_1,x_2)&=\nu 5\pi^2 y_{\rm exa}(x_1,x_2) +y_{\rm exa}(x_1,x_2)+a(t,x_1,x_2)y_{\rm exa}(x_1,x_2)\notag\\&\quad+b_1(t,x_1,x_2)g_1(x_1,x_2) +b_2(t,x_1,x_2) g_2(x_1,x_2)\notag\\&\quad+\eta\gamma^{-1}t^\gamma 5\pi^2 y_{\rm exa}(x_1,x_2)\\
\mbox{with}\quad&g_1(x_1,x_2) \coloneqq-\pi\sin(\pi x_1)\cos(2\pi x_2)-2\pi\sin(2\pi x_1)\\
\mbox{and}\quad&
g_2(x_1,x_2) \coloneqq-2\pi\cos(\pi x_1)\sin(2\pi x_2).
\end{align}
\end{subequations}

So we check our discretization by solving the analogue of~\eqref{sys-num-fullydiscrete}, for the free-dynamics, with the external forcing~\eqref{fexa}. Using a Crank--Nicolson approximation for~\eqref{fexa} we find
\begin{subequations}\label{sys-num-fullydiscrete-free}
\begin{align}
&\bfX_{j+1}^+{\underline y}_{j+1}^{\rm num}= \bfX_{j}^-{\underline y}_{j}^{\rm num}
  -k\bfE^b_j-k\overline\bfE^\xi_j+k(\underline f^{\rm exa}(t_j)+(\underline f^{\rm exa}(t_{j+1})),
  \qquad  j\in\bbN_+;\\
  &{\underline y}_{0}^{\rm num}={\underline y}^{\rm exa};
 \end{align} 
 \end{subequations}
 where~$\bfX_{j+1}^+$, $\bfX_{j}^-$, $\bfE^b_j$, and~$\overline\bfE^\xi_j$ are as in~\eqref{sys-num-fullydiscrete}.
In Fig.~\ref{fig:numerror-exa} we see the evolution (of the base-2 logarithm) of the numerical error decreases as the spatio-temporal mesh is refined, for three values of~$\eta$.
We see that the plotted errors decrease as the mesh is refined. Further, the figures suggest a quadratic convergence to the exact solution.
\begin{figure}[htbp]
\centering
\subfigure
{\includegraphics[width=0.310\textwidth]{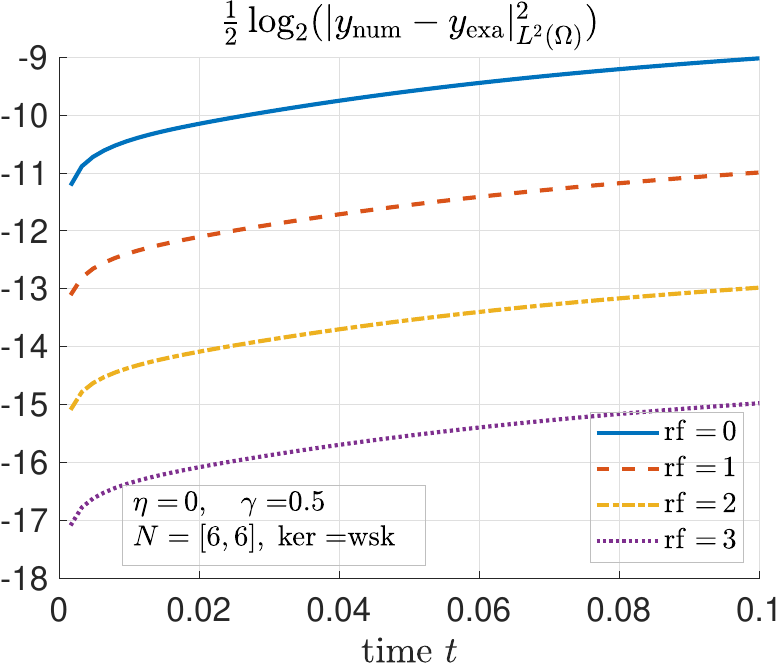}}
\;
\subfigure
{\includegraphics[width=0.310\textwidth]{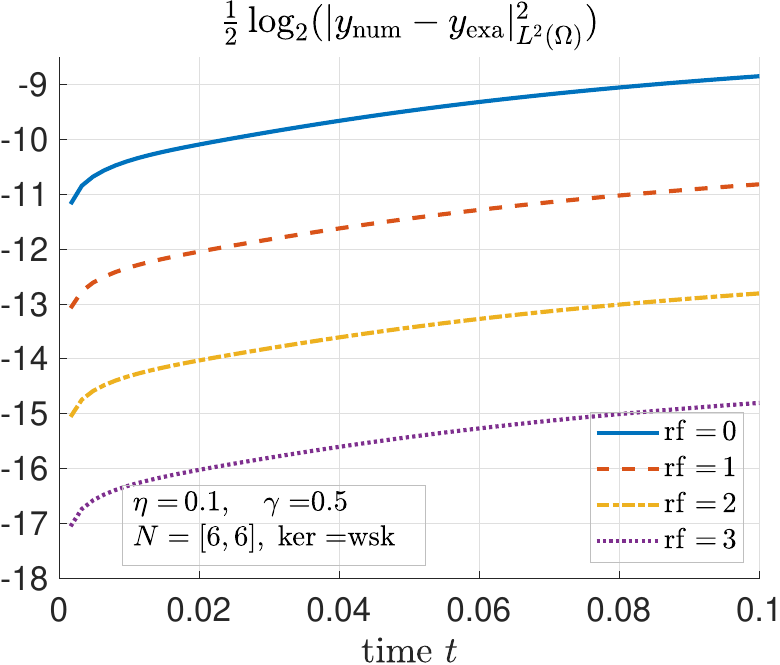}}
\;
\subfigure
{\includegraphics[width=0.310\textwidth]{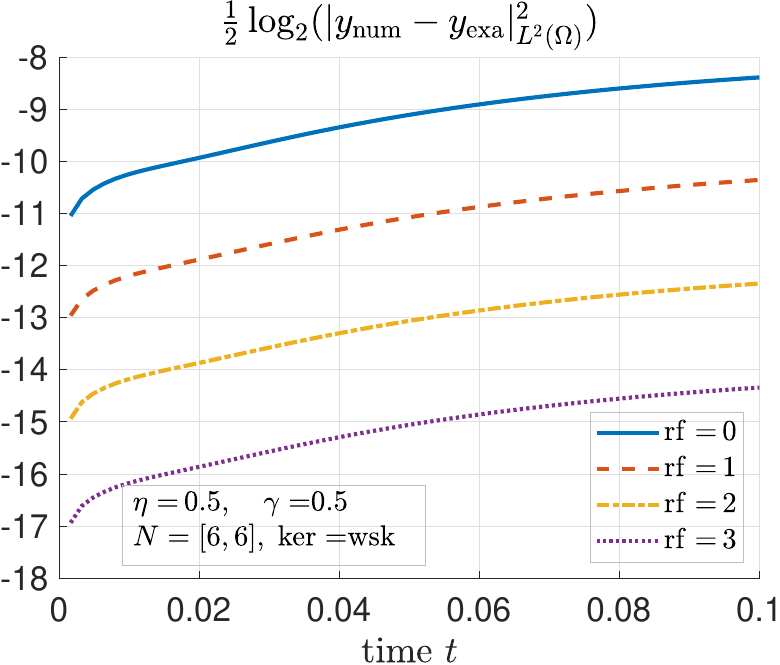}}
\caption{Mesh refinements. Numerical error for the free dynamics~\eqref{sys-num-y-2}.}\label{fig:numerror-exa}
\end{figure}

\subsubsection{Free dynamics: the closed-loop solution}
In Fig.~\ref{fig:num-free-CL-refs},
\begin{figure}[htbp]
\centering
\subfigure
{\includegraphics[width=0.310\textwidth]{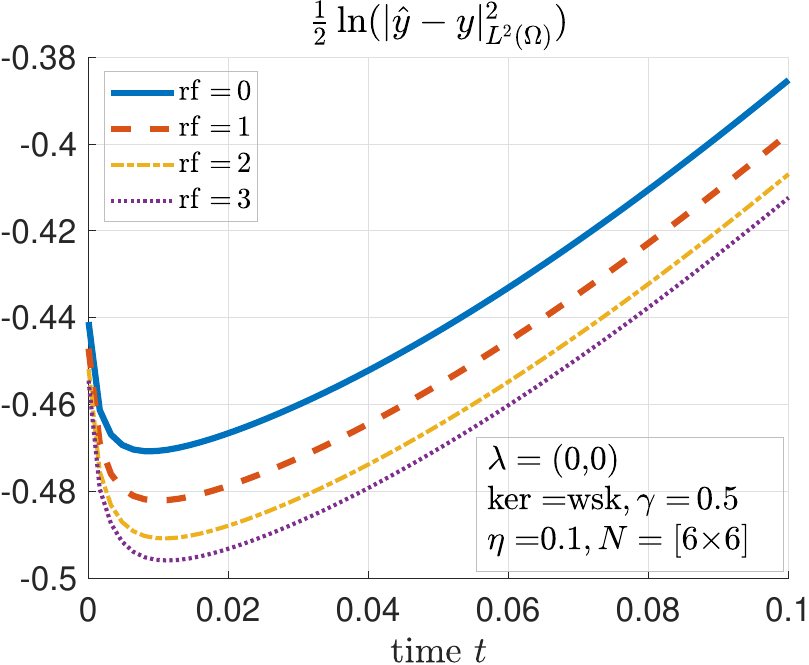}}
\qquad
\subfigure
{\includegraphics[width=0.310\textwidth]{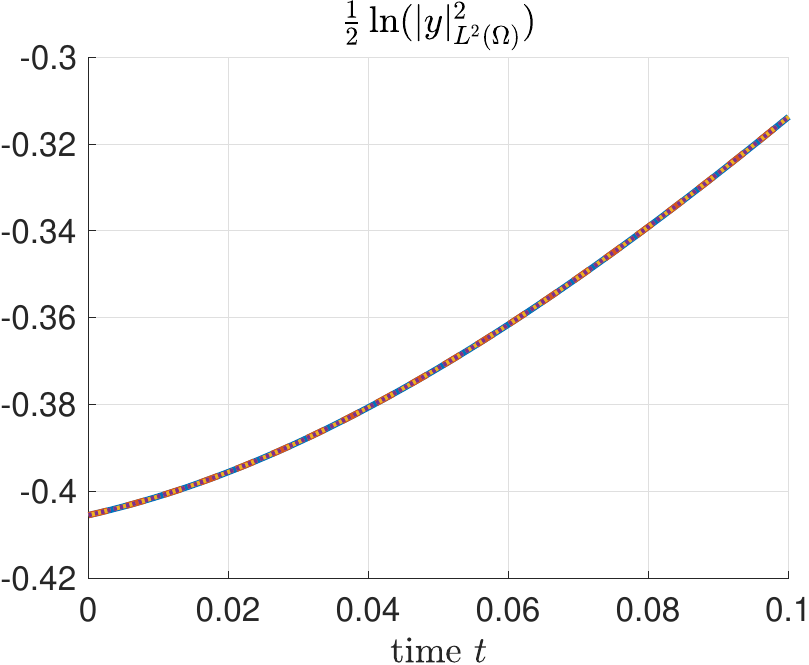}}
\caption{Mesh refinements. Free closed-loop dynamics.}\label{fig:num-free-CL-refs}
\end{figure}
 for the case~$\eta=0.1$,
we see the evolution of the norms for the solution, with zero control and injection forcings, computed in refined spatio-temporal meshes~$\fkM_{\rm rf}^{\ell\times\ell}$, ${\rm rf}\in\{0,1,2,3\}$ and~$\ell=6$ as in Section~\ref{ssec:num-free22}. We observe that the solution in the coarsest mesh~$\fkM_{0}^{\ell\times\ell}$ gives us already a good picture of the qualitative behavior of the evolution of the norms of the estimate error~$\widehat y(t)-y(t)$ and of the uncontrolled state~$y(t)$. Indeed, with the naked eye, we can hardly see a difference between the evolution of the norm of~$y$ for the three numerical solutions. The difference that we see for the evolution of the norm of~$\widehat y-y$ is due to the fact that the initial condition~$\widehat y_0=P_{\clW_S}^{\clW_S^\perp}y_0$, as in~\eqref{num-param-ic}, for the observer was computed numerically; still the behavior of the evolution of the norm of~$\widehat y-y$ in Fig.~\ref{fig:num-free-CL-refs} is similar for all refinements, with a similar instability rate.

Note that we have chosen~$\widehat y_0=P_{\clW_S}^{\clW_S^\perp}y_0$, simply to use the available output at initial time~$t=0$. Of course, we can neglect this output and simply take any other ``guess''~$\widehat y_0$ to initialize the observer. For example, if we simply take~$\widehat y_0=0$, we obtain the results in Fig.~\ref{fig:num-free-CL-refs-hatyic0}.
\begin{figure}[htbp]
\centering
\subfigure
{\includegraphics[width=0.310\textwidth]{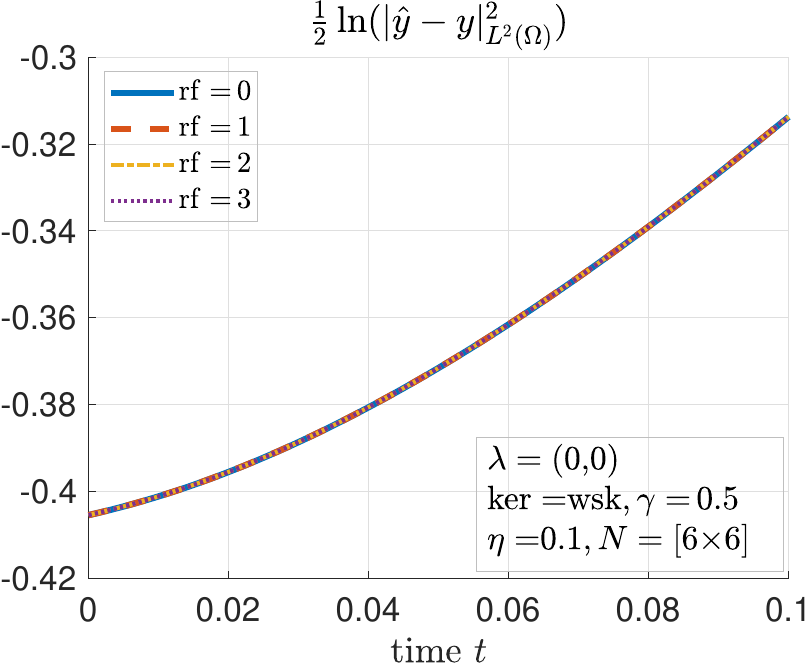}}
\qquad
\subfigure
{\includegraphics[width=0.310\textwidth]{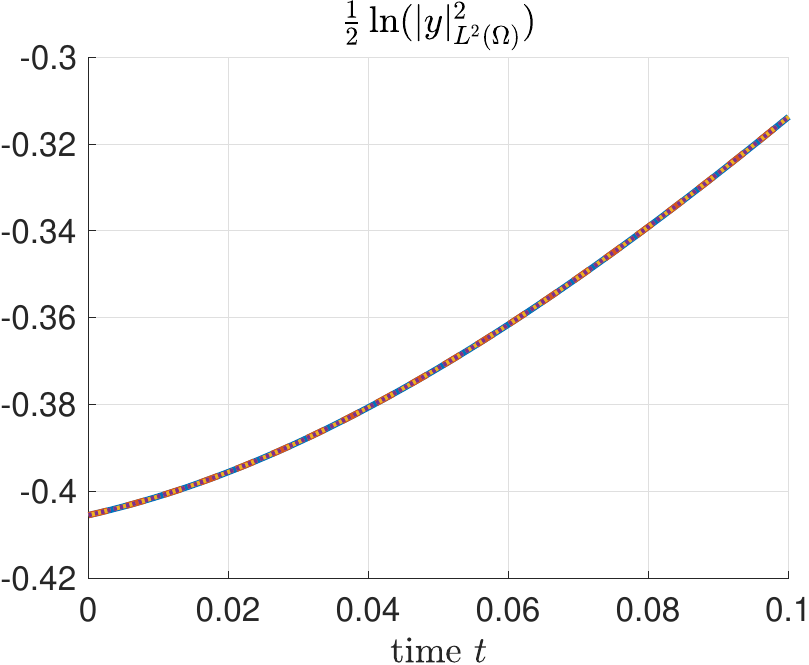}}
\caption{Mesh refinements. Free closed-loop dynamics with~$\widehat y_0=0$.}\label{fig:num-free-CL-refs-hatyic0}
\end{figure}
Now, with the naked eye, we cannot see a difference between the evolution of the norm of~$y$ for the three numerical solutions. Note that, now the numerical approximation of~$\widehat y_0=0$ does not involve any numerical computations.

\subsubsection{Controlled-injected dynamics: the closed-loop solution}
Here we take~$\ell=6$ and the corresponding to 18 sensors and 18 actuators as in Section~\ref{ssec:num44}. Now, in Fig.~\ref{fig:num-66-CL-refs},
\begin{figure}[htbp]
\centering
\subfigure
{\includegraphics[width=0.310\textwidth]{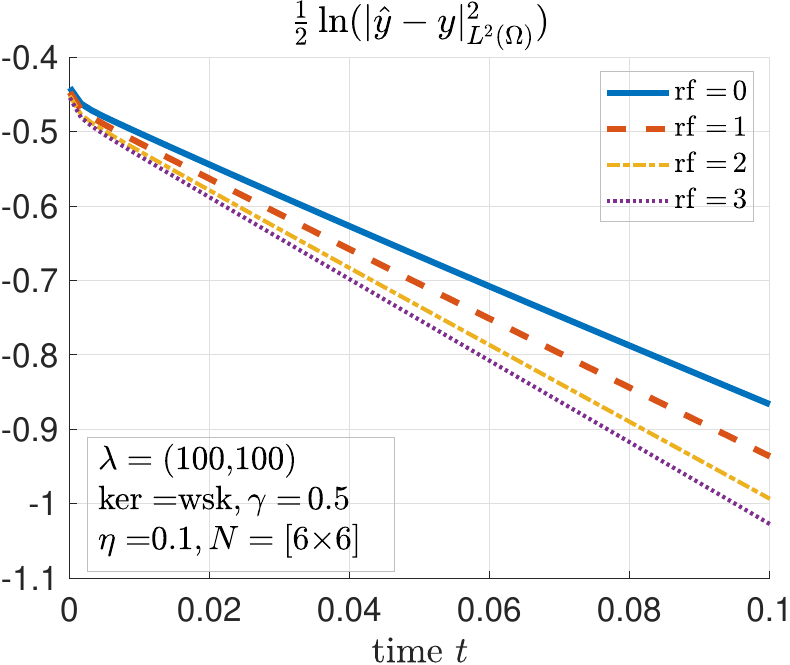}}
\qquad
\subfigure
{\includegraphics[width=0.310\textwidth]{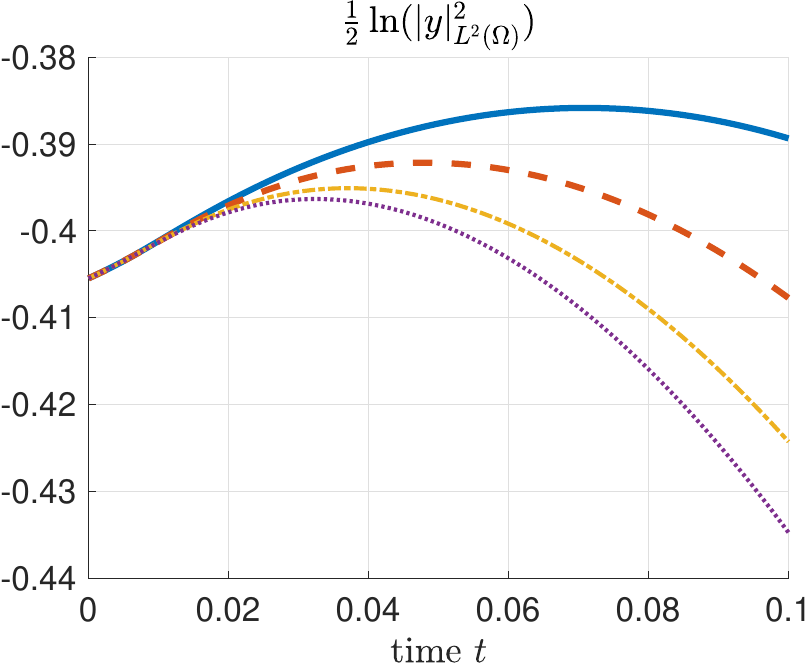}}
\caption{Mesh refinements. Controlled-injected closed-loop dynamics. }\label{fig:num-66-CL-refs}
\end{figure}
 for the case~$\eta=0.1$, we still see that the solution in the coarsest mesh~$\fkM_{0}^{\ell\times\ell}$ already gives us a good picture on the qualitative behavior of the evolution of the same norms.
Again, the results in Fig.~\ref{fig:num-66-CL-refs} correspond to the case where the initial condition~$\widehat y(0)=\widehat y_0=P_{\clW_S}^{\clW_S^\perp}y_0$ was taken for the observer. For the sake of completeness in Fig.~\ref{fig:num-66-CL-refs-hatyic0} we present the analogue of~\ref{fig:num-66-CL-refs-hatyic0} where we have taken~$\widehat y_0=0$ as our ``guess'' for the initial state.
\begin{figure}[htbp]
\centering
\subfigure
{\includegraphics[width=0.310\textwidth]{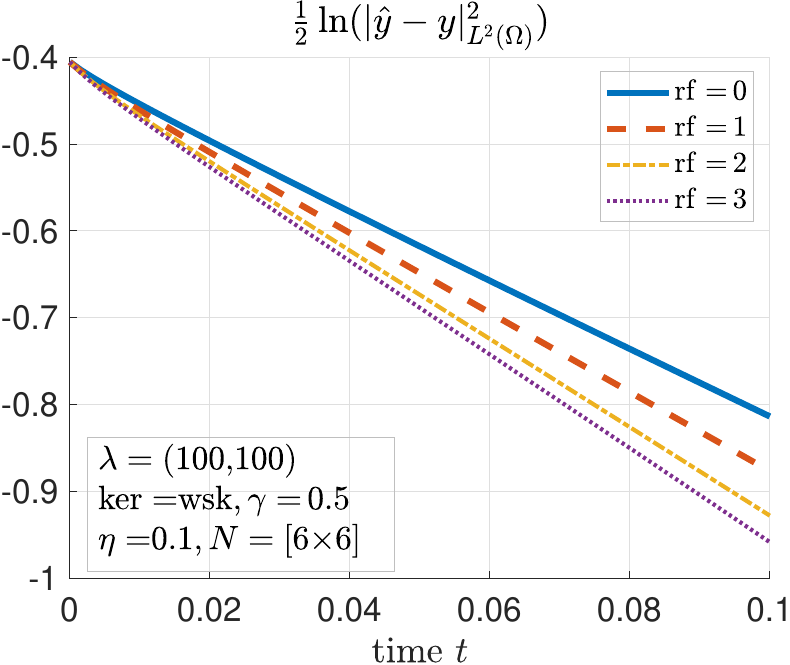}}
\qquad
\subfigure
{\includegraphics[width=0.310\textwidth]{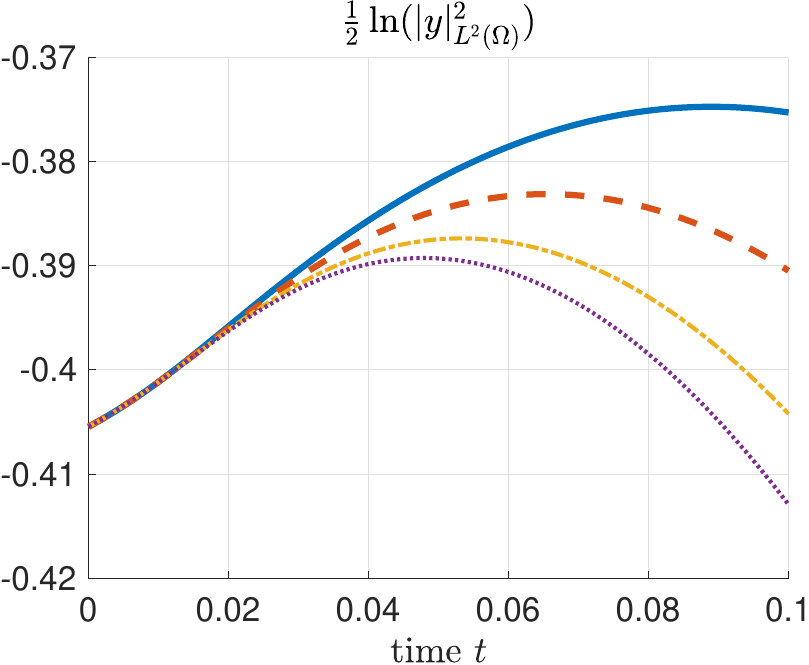}}
\caption{Mesh refinements. Controlled-injected closed-loop dynamics with~$\widehat y_0=0$.}\label{fig:num-66-CL-refs-hatyic0}
\end{figure}
Finally, note the convergence observed, in Figs.~\ref{fig:num-free-CL-refs}--\ref{fig:num-66-CL-refs-hatyic0}, for the evolution of the norms, as we refine the spatio-temporal mesh, thus showing that the numerical solution in the coarsest mesh is already able to capture satisfactorily the essential stability properties of the system.

\section{Static output-based feedback stabilizability}\label{sec:finrem-static-output}
We have focused on the general case where the actuators and sensors may correspond to different devices. Here, we consider the case where each actuator device also incorporates a sensor giving us the output~$U_M^\vee y$. In this case,
we do not need the observer to compute a state estimate. Indeed, without the observer system~\eqref{sys-num} reads
\begin{align}\label{sys-num-y-static}
\tfrac{\partial}{\partial t} {y}&=-(-\nu\Delta+\Id)y  -  a y - b \cdot\nabla y  -\eta \int_0^t \fkK(t-s)(-\Delta)y(s)ds\notag\\
&\quad +U_M^\diamond \Bigl(- \lambda_1[\clV_U]^{-1}U_M^\vee P_{\clU_M}^{\widetilde\clU_M^\perp}P_{\widetilde\clU_M}^{\clU_M^\perp} y\Bigr),
&&\quad y(0) = y_0
\end{align}
and, by recalling (the analogues of)~\eqref{output-Proj} and~\eqref{ProjOb=ObOr} we can write the input as
\begin{align}
u&=- \lambda_1[\clV_U]^{-1}U_M^\vee P_{\clU_M}^{\widetilde\clU_M^\perp}P_{\widetilde\clU_M}^{\clU_M^\perp} y
=- \lambda_1[\clV_U]^{-1}U_M^\vee P_{\clU_M}^{\widetilde\clU_M^\perp}P_{\widetilde\clU_M}^{\clU_M^\perp}P_{\clU_M}^{\clU_M^\perp} y\notag\\
&=- \lambda_1[\clV_U]^{-1}U_M^\vee P_{\clU_M}^{\widetilde\clU_M^\perp}P_{\widetilde\clU_M}^{\clU_M^\perp}U_M^\diamond [\clV_U]^{-1}U_M^\vee y\eqqcolon \clK_{\rm stat}U_M^\vee y.\label{out-2-inp-static}
\end{align}
Hence, we have that the feedback-input~$u$ depends only on the output~$U_M^\vee y$.

Observe that, in this case~$y$ satisfies the dynamics~\eqref{sys-num-y-static} analogue to that of~$z$ in~\eqref{sys-CL-Intro-error-z} by replacing~$W_S$ by~$U_M$ in~\eqref{sys-CL-Intro-error-z}, because
\begin{align}
U_M^\diamond u=- \lambda_1 P_{\clU_M}^{\clU_M^\perp}P_{\clU_M}^{\widetilde\clU_M^\perp}P_{\widetilde\clU_M}^{\clU_M^\perp}U_M^\diamond [\clV_U]^{-1}U_M^\vee y=L_M^{[\lambda_1]} U_M^\vee y.\notag
\end{align}
with~$L_M^{[\lambda_1]}$ being as in~\eqref{KL-Oper} after replacing~$W_S$ by~$U_M$.

Therefore, the stability of~\eqref{sys-num-y-static} follows by following the arguments within the proof of Theorem~\ref{thm:stab}, leading us to the analogue of~\eqref{est-z-lyap}
\begin{align}
&\norm{y(t)}{H}^2+\fkM_y(t)+2\overline\mu\int_s^t\norm{y(\tau)}{H}^2\,\rmd\tau\le\norm{y(s)}{H}^2+\fkM_y(s)\notag
\end{align}
 for general positive kernels,
and to the analogue of~\eqref{est-z-exp}
\begin{align}
&\varphi_y(t)+\overline\fkM_y(t)+2\mu\int_s^t\norm{\varphi_y(\tau)}{H}^2\,\rmd\tau\le \varphi_y(s)+\overline\fkM_y(s),\notag
\end{align}
with~$\varphi_y(t)=\rme^{-2\gamma t}$, for exponential kernels.
Summarizing, if each actuator device at our disposal also incorporate a sensor, then we have the so-called static output-based feedback stabilizability 
of~\eqref{sys-y-Intro},
\begin{align}
&\tfrac{\partial}{\partial t}{y}(x,t)= (\nu \Delta -\Id)y(x,t) -  a(x,t) y(x,t) -b(x,t)\cdot\nabla y(x,t)\notag\\
&\hspace{4em} + \eta\int_0^t \fkK(t-s)\Delta y(x,s) \mathrm{d}s +{\textstyle\sum\limits_{j=1}^{M_\sigma}}u_j(t)              \indf_{\omega^j}(x),\notag\\
&\mathfrak{B}y(x,t)|_{\partial\Omega} = 0, \quad \quad y(x,0) = y_0(x),\notag\\
&w(t)=(w_1(t),\dots, w_{M_\sigma}(t)),\quad\mbox{with}\quad w_i(t)\coloneqq \int_{\omega^i}y(x,t)\,\rmd\Omega,\quad 1\le i\le M_\sigma,\notag
\end{align}
where now the measurements are taken in the actuators domains and the input~$u(t)$, at each time~$t\ge0$, is given as a function of the output~$w(t)$, at the same time~$t\ge0$ as in~\eqref{out-2-inp-static}, $w(t)\mapsto u(t)=\clK_{\rm stat} w(t)=L_M^{[\lambda_1]} U_M^\vee y(t)$, with~$L_M^{[\lambda_1]}$ given explicitly.

Concerning real-world applications the following observation is of paramount importance, in particular, for equations with an history memory term. Recall that, with explicit feedback-input operators as in~\eqref{KL-Oper}, the computation of the input~$u(t)=K\widehat y(t)$, at time~$t$, will be performed quickly, once we have the state-estimate~$\widehat y(t)$. However, this estimate will be given by the observer which includes a memory term. For fine discretizations of the observer this can be time-consuming and can lead to nonnegligible delays, even if the output~$w(t)$ is available in real-time. In case of a static output-feedback-input as~$u(t)=K_{\rm stat} w(t)$, with explicitly given $K_{\rm stat}$, the input can be computed in real time, thus avoiding large computation times, leading to input time-delays, potentially jeopardizing the stabilizing properties of the strategy. 

\section{Final remarks}\label{sec:finalremarks}
We have shown that we can stabilize linear parabolic-like equations with memory terms involving weakly singular positive kernels. The input is constructed from an estimate~$\widehat y(t)$ provided by a Luenberger observer constructed from a given output of local measurements.  The feedback-input and output-injection operators are given explicitly, simply by scaled orthogonal projections onto the span of the actuators and sensors, respectively. The number of actuators and sensors need to be large enough depending on suitable norms of the operators defining the (unstable) free dynamics. The appropriate placement of the actuators and sensors is also given explicit for a general convex polygonal domain.

Hereafter, we include some more observations, including a discussion on subjects which could be interesting for future works.

\subsection{On the memory terms}\label{ssec:finrem-Mem}

For general kernels of positive type, we have shown that Lyapunov stability  of the closed-loop system can be achieved by taking a large enough number of actuators and sensors. For exponential kernels we have shown that we can even achieve exponential stability. 
It would also be interesting to know whether asymptotic stability can be proven for more general kernels as suggested by the simulation results in Section~\ref{sec:num-sim-wsk} for weakly singular kernels. If so, it would be also interesting to know whether a decrease rate (e.g., exponential or polynomial) can be proven as well. These are interesting questions for a potential future work.

From the numerics perspective, we have used explicitly the particular form of the memory kernel in our discretization. For general memory terms, where the explitly integration of the kernel (namely, as in~\eqref{int-o0ws} and~\eqref{int-o1ws}) is not straightforward we can still use other first-order approximations in time of the integrand of the memory term, or variants of  zero-order approximations as in~\cite[Sect.~3.2]{MahajanKhan24} and~\cite[Sect.~2.2.2]{MahajanKhanMohan24}.

In either case, the computation of the memory term can be very expensive for finer spatio-temporal meshes, requiring large amounts of computation-time and of computer memory; in particular, note that we need to save all the past history. 
\begin{itemize}
\item In order to reduce the amount of  needed computation-time and computer-memory, so that we can perform the computations in longer time intervals, it would be interesting to investigate whether saving the solution (of the observer) at a smaller (sparse) number of time instants will still provide us with a stabilizing output-based control-input. Here, the approach in~\cite{SloanThomee86} could play a key role; see also the discussions in~\cite[Sect.~4]{LeRouxThomee89} and~\cite[Sect.~5]{BakaevLarssonThomee98}.
\item In order to reduce the  needed computation-time, it would also be interesting to know if computing the solution (of the observer) in a given mesh, will still give us a stabilizing output-based control-input (say, when the controlled state is computed  in refined meshes ``mimicking/approximating'' continuous space-time). Of course, the given coarse mesh should be fine enough so that the corresponding solutions capture 
key properties of the evolution of the closed-loop coupled system (cf. Figs.~\ref{fig:num-free-CL-refs}--\ref{fig:num-66-CL-refs-hatyic0}).
\end{itemize}

\subsection{On nonlinear dynamics}\label{ssec:finrem-Nonlin}
It is well known that the investigation of closed-loop coupled systems is a nontrivial task in the case of nonlinear dynamics, because the so-called separation principle does not hold in general.
However, at a first step we could assume that the control is fixed and focus on the state-estimation problem alone. Or, we could assume that the state is available and focus on the state-stabilization problem alone. We have seen that, in the case of linear  parabolic dynamics, the memory term has a relevant effect in the evolution of the norms of the error~$\widehat y- y$ and of the state~$y$. Then, by considering the state-stabilization problem alone, it could be interesting to investigate this effect in the case of semilinear parabolic dynamics (e.g., involving additional nonlinear terms  given either by polynomial nonlinearities~$\pm y(y-1)$ and~$\pm y(y-1)(y+1)$ as in Fisher and Schl{\"o}gl equations; or monotone polynomial-like nonlinearities as~$\pm y\norm{y}{\bbR}$; or convection/Burgers-like nonlinearities involving the gradient of the state).

\bigskip\noindent
\textbf{Acknowledgments.} S. Mahajan expresses gratitude to the Ministry of Education, Government of India, for financial support through the Prime Minister Research Fellowship (PMRF ID: 2801816), which enabled the conduct of his research. The authors also thank Nagaiah Chamakuri and Asha~K. Dond as the organizers of the LACAM 2024 conference, held at IISER Thiruvananthapuram, Kerala, India, during which the initial discussions regarding this work took place.

\bibliographystyle{abbrvnat} 
\bibliography{CL-memory}

\end{document}